\documentclass[reqno]{amsart}

\usepackage{verbatim} 
\usepackage{amsmath}
\usepackage{amsfonts}
\usepackage{amssymb}
\usepackage{mathrsfs}
\usepackage{amsthm}
\usepackage{newlfont}
\usepackage{mathtools}

\usepackage{fullpage,amsmath, amssymb,amscd}
\usepackage{latexsym, epsfig, color}
\usepackage[mathscr]{eucal}
\usepackage{
graphicx}
\usepackage{amscd}
\usepackage[all, cmtip]{xy}

\usepackage[textwidth=50,textsize=tiny]{todonotes}
\usepackage{colonequals}
\usepackage{xcolor}

\newcommand\cyr{%
 \renewcommand\rmdefault{wncyr}%
 \renewcommand\sfdefault{wncyss}%
 \renewcommand\encodingdefault{OT2}%
\normalfont\selectfont} \DeclareTextFontCommand{\textcyr}{\cyr}

\newtheorem{theorem}{Theorem}
\newtheorem{lemma}[theorem]{Lemma}
\newtheorem{corollary}[theorem]{Corollary}
\newtheorem{proposition}[theorem]{Proposition}
\theoremstyle{remark} \newtheorem{remark}[theorem]{Remark}

\theoremstyle{definition}
\newtheorem{definition}[theorem]{Definition}

\def\Z{\mathbb Z}
\def\N{\mathbb N}

\def\Frob{\operatorname{Frob}}

\def\mod{\operatorname{mod}}

\def\exp{\operatorname{exp}}

\def\log{\operatorname{log}}

\def\Spec{\operatorname{Spec}}
\def\Spf{\operatorname{Spf}}

\begin{document}

\title[Universal $p$-adic sigma functions]{On the universal $p$-adic sigma and Weierstrass zeta functions}

\author[C. Blakestad]{Clifford Blakestad}
\email{math.blakestad@gmail.com }

\author[D. Grant]{David Grant}
\address{Department of Mathematics\\ University of Colorado Boulder\\
Boulder, CO 80309-0395 USA}
\email{grant@colorado.edu}

\date{\today}

\begin{abstract}
For primes $p>3$ we produce a new derivation of the universal $p$-adic sigma function and $p$-adic Weierstrass zeta functions of Mazur and Tate for elliptic curves with good ordinary or multiplicative reduction by a method that highlights congruences among coefficients in Laurent expansions of elliptic functions, and works simultaneously for generalized elliptic curves  defined by Weierstrass equations. 
\end{abstract}



\thanks{The first-named author was partially supported by NRF 2018R1A4A1023590 and NRF 2017R1A2B2001807}
\maketitle

\section{Introduction.}

Let $A$ be a complete discrete valuation ring with uniformizer $\pi$ of residue characteristic $p>2,$ and $E$ an elliptic curve over $A$ with good ordinary or multiplicative reduction modulo $\pi$. In the 1980s Mazur and Tate introduced a ``$p$-adic sigma function $\sigma_{E/A}$'' defined on the kernel of reduction of $E$ modulo $\pi,$ which shares many of the function-theoretic properties of the classical complex-valued sigma function.
It is a power series in one variable over $A,$ 
which they used to compute $p$-adic local heights of points on elliptic curves
in their investigations of the $p$-adic Birch Swinnerton-Dyer Conjecture [MST], [MTT].


The details of the construction appeared in a 1991 paper [MT]. In it they defined division polynomials for arbitrary isogenies of elliptic curves. 
They then constructed $\sigma_{E/A}$ using limits of normalized division polynomials for the isogenies $E_n\rightarrow E$ dual to the isogeny $E\rightarrow E_n$ gotten by modding out $E$ by the $p^n$-torsion in the kernel of reduction modulo $\pi$.
They also gave a multitude of equivalent conditions that uniquely characterize $\sigma_{E/A}$ (see \S3).


This circle of ideas has attracted the attention of a number of authors.
Independently, using an idea he attributed to Mumford, Norman used algebraic theta functions to construct essentially the same function \cite{N}.  
His construction worked for ordinary abelian varieties of any dimension. Norman also recognized his function as one of a class constructed earlier by Barsotti and Cristante \cite{Cr1}
 (but one that satisfies an integrality condition). Simultaneously, Cristante himself used his earlier work directly to produce integral theta functions \cite{Cr2}.
 Mazur and Tate provide references to earlier related results, and interpret the existence of $\sigma_{E/A}$  in terms of biextensions of $E\times E$ by $\mathbb G_m$ and the cubical structures of Breen [Br]. An alternative interpretation of $\sigma_{E/A}$ for $A$ an extension of the $p$-adic numbers was given by Balakrishnan and Besser, who showed that the logarithm of  $\sigma_{E/A}$ is a Coleman function [BB]. When $A$ has characteristic $p$, Papanikolas gave a different explicit formula for $\sigma_{E/A}$ [P].

Mazur and Tate also showed that their construction carried over to more general base schemes, and having done so, could 
be used to define a  ``$\sigma$-functor'' for ordinary elliptic curves over the category of formal adic schemes for which $p$ can be taken as an ideal of definition, 
uniquely determined by being compatible with base change, and by recovering their construction above for elliptic curves  
over complete DVRs with good ordinary reduction. 

For understanding such an important function, one can never have too many arrows in one's quiver.  
The goal of this paper is come up with a different construction of a ``universal $p$-adic sigma function," 
a power series attached to a generic Weierstass equation, that specializes to produce 
$\sigma_{E/A}$ for any elliptic curve with good ordinary or multiplicative reduction over a complete DVR $A$ with residue characteristic $p>3$.


To motivate our construction, we recall one of Mazur and Tate's equivalent formulations of
$\sigma_{E/A}$: Let $A$ be a complete DVR of characteristic 0 and residue characteristic $p>3$, $F$ its field of fractions, and $E$ be given by a Weierstrass model $y^2=x^3+a_4 x+a_6$, $a_4, a_6\in A$. Let  $t=-x/y$,  a parameter at the origin $O$ on the generic fibre of $E$, and $D$ the $F$-derivation on the function field of $E$ determined by $D(x)=2y$. Then expanding $x$ at $O$, standard calculations (see \cite{Si} IV.1) show one can consider $x$ as an element of $A((t))$, the ring of Laurent series in $t$ with coefficients in $A,$ and that $D$ extends to an $A$-derivation of $A((t))$.
Then $\sigma_{E/A}$ is the unique odd power-series in $t$ over $A$ whose lead term is $t$, such that
$D(D(\sigma_{E/A})/\sigma_{E/A})$
is $-x$ plus an element in $A$. 
(This characterization was the basis of Algorithm 3.1 in \cite{MST}.)

We take this as our starting point. Let $p>3$ be prime. For independent indeterminates $A_4$ and $A_6,$ let $\cal E$ be the projective closure of the curve 
given by $$y^2=f(x)=x^3 + A_4 x + A_6$$ over $R=\mathbb Z[\frac{1}{6}][A_4,A_6].$ 
Let $H$ be the coefficient of $x^{p-1}$ in the expansion of $f(x)^{(p-1)/2}$, which 
reduces to the Hasse-invariant of $\cal E$ modulo $p$ on the locus where it is elliptic. 

Let  $R_H=R[{1\over H}]$ and $\hat{R}=\underset{n}{\varprojlim}~R_H/p^nR_H$ be its $p$-completion. 
We show in \S2 that $\cal E$, now considered as a curve over $\hat{R}$,  defines a generalized elliptic curve with at worst nodal fibres [Co1], which is ordinary where elliptic (in short, a ``Weierstrass ordinary generalized elliptic curve") and we show in fact that $\cal E$ is the universal Weierstrass ordinary generalized elliptic curve over $p$-complete rings.  
(We say a ring $B$ is $p$-complete if the natural map $B\rightarrow \underset{n}{\varprojlim}~ B/p^nB$ is an isomorphism.)

Let $\hat{K}$ be the fraction field of $\hat R$. Also let  $t=-x/y$,  a parameter at the origin $O$ on the generic fibre of $\cal E$, and $D$ be the $\hat K$-derivation on the function field of $\cal E$ determined by $D(x)=2y$. As above, expanding $x$ at $O$, one can consider $x$ as an element $x(t)$ of $\hat R((t))$, the ring of Laurent series in $t$ with coefficients in $\hat R,$ and one can show that $D$ extends to an $\hat R$-derivation of $\hat R((t))$ (see \S2 for details). Let  $\hat{R}[[t]]$ denote the ring of power series in $t$ with coefficients in $\hat{R}.$ The same standard calculations show
that $Dt$ is invertible in $\hat{R}[[t]]$, and we set $W(t)=1/Dt.$

We will construct
 the universal  $p$-adic sigma function $\sigma_{\cal E/\hat{R}}(t)$ attached to $\cal E/R$, which is
the unique power series in $\hat{R}[[t]]$, odd under $t\mapsto -t$, and with lead term $t$, such that $D(D\sigma_{\cal E/\hat{R}}(t)/\sigma_{\cal E/\hat{R}}(t))+x(t)\in \hat{R}.$ 
The logarithmic derivative $D\sigma_{\cal E/\hat{R}}(t)/\sigma_{\cal E/\hat{R}}(t)$ 
will be the ``universal $p$-adic Weierstrass zeta function''  $\zeta_{\cal E/\hat{R}}(t)$.

In practice, we work in the opposite direction, constructing $\zeta_{\cal E/\hat{R}}(t)$ first. In brief detail, let $\hat{E}$ be the elliptic curve which is the basechange of $\mathcal E$ to $\hat{K}$. 
in Proposition \ref{betaprop}, for all $n\geq 1$, we study the unique function $z_n$ on $\hat{E}$, which is regular except at the origin, and whose expansion there is of the form $t^{-p^n}+H_n/t+I_n,$ for some $H_n\in \hat{K}$, $I_n\in \hat{K}[[t]]$. We show in fact that this expansion lies in $\hat{R}((t))$. Note that $z_1$ $\mod{p}$ was central in Hasse's study of his now eponymous invariant \cite{Has}, which is also given by $H_1\mod{p}$ (see also 
\cite{Vo2}.) Since $H=H_1\mod{p},$ $H_1$ is invertible in $\hat{R}$, and from that one can show that all $H_n$ are invertible in $\hat{R}$. Letting
$\zeta_n=H_n^{-1}(t^{-p^n}-z_n)$, we show that $D(\zeta_n)+x$ is congruent mod $p^n$ to some constant in $\hat{R}$, 
so if $\zeta$ is the term-by-term limit of the $\zeta_n$, it is not hard to show that it is the universal $p$-adic Weierstrass zeta function.
We note that the uniqueness of $z_n$ shows it is an odd function on $\hat{E}$, and so $\zeta_{\cal E/\hat{R}}(t)$ is odd in $t$.



Now set $\tilde{\zeta}_{\cal E/\hat{R}}(t)=\zeta_{\cal E/\hat{R}}(t)-Dt/t$, which lies in $\hat{R}[[t]]$.
Let $\Lambda(t)$ denote the integral with respect to $t$ of $\tilde{\zeta}_{\cal E/\hat{R}}(t)W(t)$ in $\hat{R}[[t]]\otimes \mathbb Q$ that has no constant term. Then if we set $\tilde{\sigma}(t)=\exp(\Lambda(t))$,  we get an even power series in $t$ with constant term 1 and with coefficients in $\hat{R}\otimes \mathbb Q$. If we then define $\sigma_{\cal E/\hat{R}}(t)=t \tilde{\sigma}(t)$, an odd power series in $t$ with lead term $t$, a calculation shows that $D\sigma_{\cal E/\hat{R}}(t)/\sigma_{\cal E/\hat{R}}(t)=\zeta_{\cal E/\hat{R}}(t)$, and the main goal of the paper is to show that $\sigma_{\cal E/\hat{R}}(t)$ actually has coefficients that lie in $\hat{R}$.
We will do that by using a version of Hazewinkel's  functional equation lemma applied to $\tilde{\sigma}(t)$ (see Corollary \ref{FunEqCor}).
That requires two things:

I) We need an endomorphism $\alpha:\hat{R}\rightarrow \hat{R}$ that lifts the Frobenius on $\hat{R}/p\hat{R}.$
We achieve this by finding a canonical subgroup $C$ of order $p$ in $\hat{E},$ and writing down a Weierstrass model for $E'=\hat{E}/C$ of the form $y'^2=x'^3+A_4'x+A_6'$, normalized so that if $\phi$ is the natural isogeny from $\hat{E}$ to $E'$, and $\omega'=dx'/2y'$, then
$\phi^*(\omega')={p\over H}\omega.$ We show in Proposition \ref{moddingout} that $A_4', A_6'\in\hat{R},$ and that we can take $\alpha$ to be the endomorphism on $\hat{R}$ induced by $A_4\rightarrow A_4', A_6\rightarrow A_6'$. 

II) We need a functional equation for $\Lambda(t),$ which we obtain in Proposition \ref{zetas} by proving that 
$$\Lambda(t)-{1\over p}\alpha(\Lambda)(t')\in \hat{R}[[t]],$$
where $t'=-x'/y'$, and we extend $\alpha$ to a map on $\hat{R}[[t]]\otimes \mathbb Q$ by acting on coefficients.

These construction are given in \S2. We will also need to verify in \S3 that $\sigma_{\cal E/\hat{R}}(t)$ universally satisfies at least one of the other equivalent characterizations of  $\sigma_{E/A}$ given by Mazur and Tate, to guarantee that $\sigma_{\cal E/\hat{R}}(t)$
specializes to $\sigma_{E/A}$ when $A$ is an equicharacteristic complete DVR as well. 
Once having done this, it is a formality in \S4 to verify that $\sigma_{\cal E/\hat{R}}(t)$ can be used
to recover the $\sigma$-functor of Mazur and Tate.

Our motivation for finding a different approach to the construction of 
 $p$-adic sigma functions was to provide a potential
path to generalizations to curves of higher genus and abelian varieties of higher dimension. Indeed some of this work --- 
in many ways more hands-on than [MT] --- proved useful in the PhD thesis of the first-named author, 
who constructs the universal $p$-adic sigma function for jacobians of curves of genus 2
in a form amenable for calculation [Bl].

We would like to thank the referee for numerous helpful suggestions that greatly improved the exposition of these results.

While this paper was in revision, our friend John Tate passed away. It is our honor to dedicate this paper to his memory.


\section{Preliminaries and statements of results}
\begin{center}\item\subsection{On WOGECs}\end{center}


Let $p>3$ be a prime. All rings will be commutative with identity.

As in the Introduction,
let $A_4$ and $A_6$ be independent indeterminates over $\mathbb Z$. 
Let $\cal E$ be the projective closure of
\begin{equation}\label{GenEqn}
y^2=f(x)=x^3+A_4x+A_6,
\end{equation}
over $R=\mathbb Z[{1\over 6}][A_4,A_6]$. 
We standardly set $C_4=-48A_4,$ $C_6=864A_6,$ and $\Delta=-16(4A_4^3+27A_6^2)=(C_4^3-27C_6^2)/1728.$
We let $H$ be the coefficient of $x^{(p-1)/2}$ in $f(x)^{(p-1)/2}.$  

To fix notation, for a given ring $A$, we will often specialize (\ref{GenEqn}) to an equation
\begin{equation}\label{SpeEqn}
y^2=x^3+a_4x+a_6,~a_4, a_6\in A,
\end{equation}
and let $\delta,$ $c_4$, $c_6,$ and $h$ denote the corresponding specializations of $\Delta$, $C_4$, $C_6$, and $H$.   


Indeed, one case we will consider is that $A$ is a complete discrete valuation ring with field of fractions $F$, and (\ref{SpeEqn}) is a minimal model over $A$ of an elliptic curve $E$ over $F$ which has multiplicative reduction. In that case (\ref{SpeEqn}) is not elliptic over $A$, but rather an example of a generalized elliptic curve over $A$. \footnote{We refer the reader to \cite{Co1} and \cite{Co2} for background on generalized elliptic curves. Following \cite{Co2}, a stable genus-1 curve $X$ over a scheme $S$ is a scheme that is proper, smooth, and of finite presentation over $S$, with all its fibres over geometric points being smooth curves of genus 1 or N\'eron $n$-gons. We let $X^{sm}$ denote its smooth locus. A generalized elliptic curve $E/S$ is a stable genus-1 curve over $S$, along with a section $e \in E^{sm}(S)$, and a map $+:E^{sm}\times E\rightarrow E$ such that $+$ restricts to turn $(E^{sm},e)$ into a commutative group scheme with cyclic geometric component group.}

\begin{remark} We do not require anything from the theory of generalized elliptic curves, but the following is motivation for the definition below: (I) a Weierstrass cubic over a ring $A$ whose geometric fibres are either elliptic curves or nodal cubics is a generalized elliptic curve. If $6$ is invertible in $A$, we can change models to put it in the form (\ref{SpeEqn}); (II) \cite{Co2} explains that the Riemann-Roch Theorem shows that any generalized elliptic curve with geometrically irreducible fibres and a choice of section through the smooth locus can be given locally on the base by a Weierstrass cubic (for details over a locally noetherian base scheme, see  \S2.25 in \cite{Hi}). 
\end{remark}






Recall we say an elliptic curve in characteristic $p$ is ordinary if its Hasse invariant is nonzero. We refer the reader to \cite{KM} 12.4, and \cite{L} Appendix 2, \S5, 
for equivalent characterizations of the Hasse invariant, one of which is (for $p>3$) that it's given by $h$ for an elliptic curve defined by an equation (\ref{SpeEqn}) over a ring of characteristic $p$.

\bigskip
\begin{definition} Let $A$ be a ring where 6 is invertible.

1) A Weierstrass generalized elliptic curve over $A$ is a 
curve over $A$ defined by a Weierstrass equation as in (\ref{SpeEqn}), all of whose fibres over geometric points are either elliptic curves or nodal cubics. Any two are said to be isomorphic over $A$ if there a unit $u\in A$ such that $a_4'=u^4a_4$ and $a_6'=u^6a_6.$

2) A Weierstrass ordinary generalized elliptic curve (WOGEC) over $A$ is a Weierstrass generalized elliptic curve over $A$ whose elliptic fibres over geometric points are ordinary.    
\end{definition}


Important examples of WOGECs are the minimal models of elliptic curves over complete DVRs
that have either good ordinary or multiplicative reduction.

For $p>3$, we will see shortly that there is a convenient way to characterize WOGECs over a ring $A$ whose closed points all have residue characteristic $p$. 
For that we recall from \cite{Si}, III, \S 1, that for a curve (\ref{SpeEqn}) defined over a field,  it is singular exactly when $\delta=0$, in which case it is a nodal cubic if and only if $c_4\neq 0$. First we need a lemma.


\begin{lemma}\label{hasse} Suppose $p>3$. Then $H^2\equiv C_4^{(p-1)/2}\mod{(p,\Delta)}$.
	
\end{lemma} 

\begin{proof} 
Since $p>3$, $2$ and $3$ are invertible in $\mathbb Z/p\mathbb Z$, so $R/(p,\Delta)=\mathbb Z[C_4,C_6]/(p,\Delta).$
Let $C_2$ be an indeterminate.
Since $1728 \Delta=C_4^3-C_6^2$, there is an injection
$\mathbb Z[C_4,C_6]/(p,\Delta)\rightarrow \mathbb Z[C_2]/(p)$ given by $C_4=C_2^2, C_6=C_2^3$. Viewing $H$ as a polynomial in $A_4$ and $A_6$, it now suffices to show that $H(-C_2^2/48,C_2^3/864)\equiv C_2^{(p-1)/2}\mod{p}.$
 But when $A_4=-C_2^2/48$ and $A_6=C_2^3/864$, $f(x)=(x-C_2/12)^2(x+C_2/6)$,
and then it is easy to verify that the coefficient of $x^{p-1}$ in $f^{(p-1)/2}$ is the same modulo $p$ as the coefficient of $x^{p-1}$ in $(x^2(x+C_2/4))^{(p-1)/2}$, which
is $(C_2/4)^{(p-1)/2}\equiv C_2^{(p-1)/2}\mod{p}$, as needed.
\end{proof}

\begin{corollary}\label{Criterion} Let $p>3$ and $A$ be a ring whose closed points all have residue fields of characteristic $p.$
Let $X$ be a Weierstrass generalized elliptic curve defined over $A$ by a model as in (\ref{SpeEqn}), and $h$ and $\delta$ the corresponding specializations from $R$ of $H$ and $\Delta$.
Then $A=(\delta,h).$ If in addition
$X$ is a WOGEC, $h$ is invertible in $A.$

Conversely, if $X$ is a projective curve over $A$ defined by a model of the form (\ref{SpeEqn}) such that $h$ is invertible in $A$, then $X$ is a WOGEC.
\end{corollary}

\begin{proof}
If the ideal $(\delta,h)$ were not the unit ideal  in $A$, there would be a maximal ideal $\mathfrak{m}$ containing it.
Let $\iota$ be an arbitrary embedding of the field $A/\mathfrak{m}$ into an algebraic closure $\overline{A/\mathfrak{m}}$.
Then $\iota(\delta)=0$ so $X$ is not elliptic over $\overline{A/\mathfrak{m}}$.
By the definition of Weierstrass generalized elliptic curve, $X$ must be a nodal cubic over $\overline{A/\mathfrak{m}}$, which means $\iota(c_4)\neq 0$ and so $c_4$ cannot be in $\mathfrak{m}$.
Together with the fact that $p$ must be in $\mathfrak{m}$ by assumption, Lemma \ref{hasse} forces $\iota(h) \neq 0$ and hence $h$ cannot be in $\mathfrak{m}$, contrary to assumption.

Now assume in addition that $X$ is a WOGEC. If $h$ were not a unit, it would be contained in a maximal ideal $\mathfrak{m}$.
By the above, $\delta$ is not in $\mathfrak{m}$.
Hence for any embedding $\iota:A/\mathfrak{m}\rightarrow \overline{A/\mathfrak{m}}$, we have $\iota(\delta)\neq 0,$ so $X$ is not elliptic over $\overline{A/\mathfrak{m}}$.
But since $\iota(h)=0$, 
$X$ is not ordinary over $\overline{A/\mathfrak{m}}$.
Thus $X$ cannot be a WOGEC unless $h$ is a unit.

Conversely, let $X$ be projective over $A$ given by equation (\ref{SpeEqn})
and suppose that $h$ is invertible in $A$. 
For those maximal ideals $\mathfrak{m}$ of $A$ containing $\delta$, Lemma \ref{hasse} implies the associated fibres of $X$ must be nodal. 
This means $X$ is a Weierstrass generalized elliptic curve.
For those maximal ideals $\mathfrak{m}$ not containing $\delta$,  the fibre of $X$ over $\mathfrak{m}$ has non-zero Hasse invariant so is ordinary, making $X$ a WOGEC.
\end{proof}

\begin{definition}
Let $R_H=R[{1\over H}]$, and $\hat{R}=\underset{n}{\varprojlim}~R_H/p^nR_H$
be its $p$-completion.
\end{definition}

From now on we will consider $\cal E$ as a scheme over $\hat{R}.$

\begin{definition}\label{rho-specialization} If $X$ is a WOGEC over a $p$-complete ring $A$ given by a model as in (\ref{SpeEqn}), and $\rho$ is a continuous ring homomorphism from $\hat{R}$ to $A$
such that $\rho(A_4)=a_4$ and $\rho(A_6)=a_6,$ then we will say that $X$ is a $\rho$-specialization of $\cal E$ and write $X={\cal E}_\rho$. 
\end{definition}

\begin{proposition}\label{specialization} 1) $\mathcal E/\hat{R}$ is a WOGEC over $\hat{R}$.

2)  $\cal E$ is the universal WOGEC over $p$-complete rings $A$ with $p>3,$ in the sense that any WOGEC over $A$ is uniquely a $\rho$-specialization of $\cal E.$
\end{proposition}

\begin{proof}
1)  Since every element in $1+p\hat{R}$ is a unit, every maximal ideal of $\hat{R}$ contains $p$. Also $H$ is invertible in $\hat{R}$ by construction, so Corollary \ref{Criterion} gives the result.

2) Let $X$ be a WOGEC over a $p$-complete ring $A$ given by a model as in (\ref{SpeEqn}). All we need to show is that there is a unique continuous ring homomorphism from $\hat{R}$ to $A$ sending $A_4\rightarrow a_4$ and $A_6\rightarrow a_6.$ There is a unique evaluation map $\rho: R\rightarrow A$ with this property, which send $H$ to $h$. Since all the geometrically closed points of $Spec(A)$ have residue characteristic $p$, by Corollary \ref{Criterion}, $h$ is invertible in $A.$ Hence $\rho$ extends uniquely to a ring homomorphism from $R_H$ to $A.$ Since $A$ is $p$-complete, $\rho$ extends uniquely  to a continuous ring homomorphism from $\hat{R}$ to $A.$
\end{proof}

\begin{center}\item\subsection{On derivations and expansions at infinity}\end{center}

The two affine schemes over $\hat{R},$ $U=\Spec(\hat{R}[x,y]/(y^2-x^3-A_4x-A_6))$, $V=\Spec(\hat{R}[t,w]/(w-t^3-A_4tw^2-A_6w^3))$ are an open cover of $\cal E,$ glued together by $t=-x/y,w=-1/y$ on their overlap. 

We will denote the $\hat{R}$-point $(0,0)$ on $V$ by $\infty$. 
We note that $\infty$ is defined on $V$ by the ideal $I_\infty=(t,w)$ in the coordinate ring $\cal O(V)$ of $V$, and (\cite{Si}, IV, \S1, Proposition 1.1, applied with $a_1=a_2=a_3=0$) shows that we can identify $\varprojlim\cal O(V)/I_\infty^n$ with $\hat{R}[[t]],$ the ring of power series in $t$ with coefficeints in $\hat R$, and we will consider $\cal O(V)$ as embedded in $\hat{R}[[t]].$ Let $\hat K$ denote the field of fractions of $\hat R$, so the function field $\hat{K}(\cal E)$ of $\cal E$ over $\hat{K}$ is the fraction field of $\cal O(U)$ or $\cal O(V).$ 
We will likewise consider $\hat{K}(\cal E)$ as embedded in $\hat{K}((t))$, the ring of Laurent series in $t$ with coefficients in $\hat{K}$, which is the fraction field of $\hat R[[t]].$



On the generic fibre of $\cal E$ (which is elliptic), there is an invariant differential given by $\omega = \frac{dx}{2y},$ which induces a $\hat K$-derivation $D$ on 
$\hat K(\cal E)$
by $D(g)=\frac{dg}{\omega}$.
Computing this on $x$ and $y$ yield $$D(x)=2y, D(y)=f'(x)=3x^2+A_4.$$
Note that $D(x)$ and $D(y)$ are in $\cal O(U),$ and that $x$ and $y$ generate $\cal O(U)$ over $\hat{R}$. Because the only relation on $x$ and $y$ is $y^2=f(x),$ and $D(x)$ and $D(y)$ are consistent with the equation $D(y^2)=D(f(x))$, $D$ restricts to an $\hat{R}$-derivation $D: \cal O(U) \rightarrow \cal O(U)$.

Similarly, one computes that 
\begin{equation}\label{tDeriv}
D(t)=(xD(y)-yD(x))/y^2=1-2A_4tw-3A_6w^2,
\end{equation}
and $$D(w)=D(y)/y^2=3t^2+A_4w^2,$$
which are consistent with the equation $D(w)=D(t^3+A_4 t w^2 + A_6 w^3),$ so $D$ restricts to an $\hat{R}$-derivation $D:\cal O(V) \rightarrow \cal O(V)$.
It follows furthermore from (\ref{tDeriv}) that $D$ extends to an $\hat R$ derivation on $\hat{R}[[t]]$, and thence to an $\hat R$ derivation on the ring $\hat{R}((t))$ of Laurent series in $t$ over $\hat{R}$.


Suppose that $\rho$ is a continuous map from $\hat R$ to a $p$-complete ring $S$. 
If $\cal E_\rho$ is the $\rho$-specialization of $\cal E$, we can view it as the base change $Spec(S)\times_{Spec(\hat{R})}\cal E$ induced by $\rho$, whose second projection we denote by $p_2$.
We can then analogously define the $S$-derivation $D_\rho$ on $\cal O(p_2^{-1}(U))$ such that
$D_\rho(x)=2y$,
which \emph{mutatis mutandis} extends to $\cal O(p_2^{-1}(V))$ and then to $S((t)).$ 
It follows that if we let $\rho$ also denote the map from $\hat R((t))$ to $S((t))$ gotten by letting $\rho$ acts on coefficients of Laurent series, then
\begin{equation}\label{Dcommutes}
D_\rho\circ \rho=\rho\circ D.
\end{equation}
By abuse of notation we will also let $D$ denote $D_\rho,$ and then abbreviate (\ref{Dcommutes}) by saying that $D$ commutes with $\rho$-specialization. In particular, $D$ will then commute with reduction mod $p\hat{R}.$

In $\hat{R}((t))$, we standardly get
that 
the expansions of $\omega/dt, x, $ and $y$ in terms of $t$ have the forms (\cite{Si}, IV, \S1, setting $a_1=a_2=a_3=0$)
\begin{equation}\label{tExpand}
{{\omega}\over{dt}}:=W(t):=\sum_{n=0}^\infty w_nt^n \in 1+t^4\hat R[[t]],
~x(t)\in{1\over{t^2}}+t^2\hat R[[t]],~ y(t)\in-{1 \over{t^3}}+t\hat R[[t]].
\end{equation}
(We note that these calculations work for Weierstrass equations defined over any ring, and do not require the discriminant of the Weierstrass equation to by invertible in the ring.)

With this we get $Dt=1/W(t)$, and hence the action of $D$ on an element $a(t)\in \hat{R}((t))$ is
$D(a(t))=\frac{1}{W(t)}\frac{da}{dt},$
and more generally for $\rho$ and $S$ as above, that for any element 
  $b(t)\in S((t))$,
\begin{equation}\label{Daction}
D(b(t))={1\over{\rho(W)(t)}}{{db}\over {dt}}.
\end{equation}

\begin{center}\item\subsection{On weights}\end{center}

We consider $R=\Z[\frac{1}{6}][A_4,A_6]$ as an $\N$-graded ring by giving elements of $\Z$ weight 0 and $A_4$ a weight of 4 and $A_6$ a weight of 6 (hence the subscripts).
These weights are specifically chosen to match their weights as modular forms (see Remark \ref{modular}).

We can then extend this grading to the ring
$R[x]$ by giving $x$ a weight of 2. Then $f(x)=x^3+A_4x+A_6$ is homogeneous of weight of 6, so the weight extends to the quotient ring $R[x,y]/(y^2-f(x))$
by giving $y$ a weight of 3. The weight then extends uniquely to its fraction field, which is then a $\Z$-graded ring, whereby $t=-x/y$ has weight $-1$. 

Note that $f(x)^{(p-1)/2}$ is homogeneous of weight $3(p-1)$, so its coefficient $H$ of $x^{p-1}$ is homogeneous of weight $p-1$. Hence the weight extends to the localization $R_H$ of $R$ which is then a $\Z$-graded ring. 


We defined the $p$-completion $\hat{R}$ of $R$ as the inverse limit as rings over $m$ of $R_H/p^mR_H$, which is not a graded ring in the weight inherited from $R_H$. However, $\hat{R}$
has a graded subring $\hat{R}_g$ which we can identify with the
inverse limit as $\Z$-graded rings of the $R_H/p^mR_H$.

To do so, for any integer $n$, let $(R_H)_n$ denote the subgroup of homogeneous elements of $R_H$ of weight $n$, and 
define the subgroup $\hat{R}_n$ of $\hat{R}$ as the inverse limit over $m$ of the groups $(R_H)_n/p^m(R_H)_n$. 
We set $\hat{R}_g=\oplus_{n\in\Z}\hat{R}_n.$ We will only use the word ``weight" to apply to an element of $\hat{R}$ if it
lies in some $\hat{R}_n$. 

Note however for any $\alpha\in \hat{R}$, it is the limit of its reductions $\alpha_m \mod p^m,$ each of which is a finite sum of its homogeneous components $\alpha_{m,n}$. Hence if we set $\beta_n:=\lim_{m\rightarrow\infty} \alpha_{m,n}\in \hat{R}_n$, and $\gamma_{T,m}:=\sum_{|n|\leq T}\alpha_{m,n}$, then for every $M$ there is an $T=T(M)$ such that $\gamma_{T,m}\equiv\alpha_m \mod p^m$ for all $m\leq M$. Hence 
if $\gamma_T:=\lim_{m\rightarrow\infty} \gamma_{T,m}=\sum_{|n|\leq T}\beta_{n},$ then $\alpha$ can be written uniquely in the form
$$\alpha=\lim_{T\rightarrow \infty} \gamma_T,$$
which shows how $\alpha$ is uniquely determined by its homogeneous components $\beta_n$.


For every $\kappa \in \Z_p^\times$, we define the grade preserving automorphism $gr_\kappa$ of $R\otimes \Z_p$ that send $A_4\rightarrow \kappa^4 A_4$ and $A_6\rightarrow \kappa^6 A_6$. Since $gr_\kappa(H)=\kappa^{p-1}H$, the map extends to $R_H\otimes \Z_p$ and thence to $\hat{R},$ since it commutes with reduction mod $p^m$.
Note that for any $n\in \Z$, 
for $\xi \in \hat{R}_n$, $gr_\kappa(\xi)=\kappa^n \xi.$ 
Now suppose that $\kappa$ is of infinite order in $\Z_p^\times.$
Using the notation above, if $\alpha \in \hat{R}$ has the property that $gr_\kappa(\alpha)=\kappa^n\alpha$ for some $n\in \Z$, then for any $n'\in \Z$,
$gr_\kappa(\beta_{n'})=\kappa^n\beta_{n'},$ so if $n'\neq n,$
$\beta_{n'}=0,$ and hence $\alpha\in\hat{R}_n.$
We will use this observation without further comment.



For $\kappa\in \Z_p^\times$, by the reasoning above, the map $gr_\kappa$
extends to an automorphism of $\cal O(U)$ by setting $gr_\kappa(x)=\kappa^2x$ and $gr_\kappa(y)=\kappa^3y.$ It then extends to an automorphism of its fraction field $\hat{K}(\cal E).$
We will define the weight $n$ elements in $\hat{K}(\cal E)$
to be the elements which get multiplied by $\kappa^n$ under $gr_\kappa$ for all $\kappa\in \Z_p^\times$.

Note that for every $\kappa\in \Z_p^\times$, since we have $gr_\kappa(t)=\kappa^{-1}t$
and $gr_\kappa(w)=\kappa^{-3}t,$ $gr_\kappa$ on $\hat K(\cal E)$  restricts to an automorphism of $\mathcal O(V),$ which then extends to a continuous automorphism of $\hat{R}[[t]].$ Likewise the automorphism $gr_\kappa$ on $\hat K(\cal E)$ extends to a continuous automorphism of $\hat{K}((t)).$


\begin{center}\item\subsection{Statements of results}\end{center}

Our goal is to show the following, originally due to Mazur and Tate (\cite{MT}, Appendix II.)

\begin{theorem}\label{theorem1}  There is a unique power series $\sigma_{\cal E/\hat{R}}(t)$ in $\hat{R}[[t]]$, odd under $t\rightarrow -t$,
and with lead term $t$, such that $D \left(\frac{D(\sigma_{\cal E/\hat{R}}(t))}{\sigma_{\cal E/\hat{R}}(t)} \right)+x(t)$ is some element $\beta \in \hat{R}$. We call  $\sigma_{\cal E/\hat{R}}(t)$
the universal $p$-adic sigma function.
\end{theorem}

\begin{theorem}\label{theorem2}  There is a unique Laurent series  $\zeta_{\cal E/\hat{R}}(t)$ in  $1/t+\hat{R}[[t]]$, odd under $t\rightarrow -t$,
such that $D(\zeta_{\cal E/\hat{R}}(t))+x(t)$ is some element $\beta \in \hat{R}$. 
We call $\zeta_{\cal E/\hat{R}}(t)$
the universal $p$-adic Weierstrass zeta function.
\end{theorem}

Given  $\sigma_{\cal E/\hat{R}}(t)$ it follows that $D(\sigma_{\cal E/\hat{R}}(t))/\sigma_{\cal E/\hat{R}}(t)
=\zeta_{\cal E/\hat{R}}(t)$, which is the order of
construction done by
Mazur and Tate. We will reverse the order by first constructing $\zeta_{\cal E/\hat{R}}(t),$
and then showing there is a unique $\sigma_{\cal E/\hat{R}}(t)\in \hat{R}[[t]]$, odd under $t\rightarrow -t$ and having lead term $t$, such that 
$D(\sigma_{\cal E/\hat{R}}(t))/\sigma_{\cal E/\hat{R}}(t)=\zeta_{\cal E/\hat{R}}(t)$.

The following is now formal since the expansions in (\ref{tExpand}) hold over any specialization of $\hat{R}$.

\begin{corollary}\label{sigma-specialization} Let $p>3$. If $X$ is a WOGEC over a $p$-complete ring $A$ given by a model as in (\ref{SpeEqn}), and $\rho$ is a continuous ring homomorphism from $\hat{R}$ to $A$
such that $\rho(A_4)=a_4$ and $\rho(A_6)=a_6$ (so $X$ is a $\rho$-specialization of $\cal E$), then
letting $\rho$ act on coefficients of Laurent series, and setting
$\sigma_{X/A}(t)=\rho(\sigma_{\cal E/\hat{R}})(t),~\zeta_{X/A}(t)=\rho(\zeta_{\cal E/\hat{R}})(t),$
we have
$D(\sigma_{X/A}(t))/\sigma_{X/A}(t)=\zeta_{X/A}(t),$
and
$D(\zeta_{X/A}(t))+x(t)=\rho(\beta),$ where $D$ acts on $A((t))$ as in (\ref{Daction}).
\end{corollary}


\begin{remark}
It may be helpful to explain the role of some of the hypotheses that go into Theorem \ref{theorem1}.
Suppose $A$ is a $p$-complete DVR of characteristic $0$, and $X$ is an elliptic curve over $A$ with ordinary or multiplicative reduction, given by a model of the form (\ref{SpeEqn}).

If $a$ is any element in the fraction field $F$ of $A$, then there is a unique odd power series $\sigma_a(t)$ in $F[[t]]$ with lead term $t$ such that
$D\left(\frac{D \sigma_a(t)}{\sigma_a(t)} \right) = -x+a$.
However, there is a unique $\beta$ (which necessarily lies in $A$) such that $\sigma_\beta(t)$ has coefficients in $A$, or even has coefficients which have bounded powers of $p$ in their denominators. (If instead $X$ had supersingular reduction, then no $\sigma_a(t)$ could have $p$-bounded coefficients. See \cite{BKY} for a discussion, especially in the supersingular case.)

Furthermore, if we relax the requirement that $\sigma_a$ be odd, for a given $a$ in $F,$ we can consider the full set of $\theta_a(t)$ in $F[[t]]$ with lead term $t$ such that
$D\left(\frac{D \theta_a(t)}{\theta_a(t)} \right) = -x+a$. But still only the $\theta_\beta(t)$ (for the same $\beta$ in $A$ as above) can have coefficients with $p$-bounded denominators. In the case that $X$ has multiplicative reduction, the $p$-adic theta functions of Tate \cite{T} give a family of such examples.
\end{remark}



\section{The constructions}

We will now carry out the constructions of $\zeta_{\cal E/\hat{R}}(t)$ and $\sigma_{\cal E/\hat{R}}(t)$, and show that they satisfy Theorem \ref{theorem2} and Theorem \ref{theorem1}, respectively.



\begin{center}\item\subsection{The construction of  the universal $p$-adic Weierstrass $\zeta$-function}\end{center}

Let $k$ be the fraction field of $\hat{R}/p\hat R$. Let $E$ be the basechange of $\cal E$ to $\hat K$ which is elliptic, and let $\mathcal E_p$ be the reduction of $\mathcal E$ over $\hat R/p\hat R,$ which is elliptic over $k$ since $\Delta$ does not vanish identically as a polynomial modulo $p$. 
%
Recall $p>3.$ 

\begin{proposition}\label{betaprop}  For any divisor $\cal D$ 
on $E/\hat K$, we standardly let $L(\cal D)$ denote the
$\hat K$-vector space of functions $f$ on $E/\hat K$ such that $(f)+\cal D$ is effective or $f=0$. Let $O$ denote the
origin on $E$.

a) For any $m\geq 2$, there is a unique element $\alpha_m$ of $L(mO)\cap \hat R[x,y]$ whose expansion in $t$ at the origin is
of the form $$t^{-m}+r_m/t+s_m t +t^3 u_m(t),$$ for some $u_m(t)\in \hat R[[t]]$,
where $r_m$ and $s_m$ are some   polynomials in $\hat R$ of weights $m-1$ and $m+1$, respectively. Its uniqueness makes $\alpha_m$ odd
if $m$ odd and even if $m$ even.

b) For any $n\geq 1$, there is a unique element $z_n$ of $L(p^nO)\cap \hat R[x,y]$ whose expansion in $t$ at the origin is 
of the form $$t^{-p^n}-H_n/t-J_n t +t^3 v_n(t),$$ for some $v_n(t)\in \hat R[[t]]$,
where $H_n$ and $J_n$ are some   polynomials in $\hat R$ of weights $p^n-1$ and $p^n+1$, respectively. Its uniqueness makes $z_n$ odd.

c) $H_{n}$ is a unit in $\hat{R}$, and $\zeta_n :=H_n^{-1}(t^{-p^n}-z_{n})$ lies in $\frac{1}{t}+ t\hat{R}[[t]]$.

d) Let $\beta_n=J_n/H_n \in \hat{R}$. Then $D\zeta_n\equiv -x+\beta_n\mod{p^n}$.
\end{proposition}

\begin{remark}\label{hasse}  
Then as noted in the Introduction, the reduction of $z_1$ modulo $p$ was studied by Hasse in his seminal paper where he introduced what is now called the Hasse invariant, one of whose incarnations is $H_1$ modulo $p$ \cite{Has}. That this agrees with what we are calling $H$ modulo $p$ is due to Deuring \cite{Deu}. For our purposes the chief take-away from this equality is that $H_1$ is invertible in $\hat{R}.$
We will need later that the coefficient of $x^{p(p-1)/2}$ in the $p$-division polynomial attached to $E$ is another element in $\hat{R}$ that reduces modulo $p$ to the Hasse invariant of $ \mathcal E_p$ \cite{Der}.
We will also use in Remark \ref{modular}
that
the polynomial in $A_4$ and $A_6$ which gives the Eisenstein series $\mathbb E_{p-1}$
when $A_4$ and $A_6$ are considered as modular forms, reduces mod $p$ to the Hasse invariant of $ \mathcal E_p$.

\end{remark}

\begin{proof} (a) We will proceed by induction. For $m \geq 2$, let $\tilde{\alpha}_m$ be $x^{\frac{m}{2}}$ when $m$ is even and $-yx^{\frac{m-3}{2}}$ when $m$ is odd.
Then $\tilde{\alpha}_m$ is an element of $L(mO)\cap \hat R[x,y]$ and from (\ref{tExpand}), we find $\tilde{\alpha}_m$ is in $R((t))$ with lead term $\frac{1}{t^m}$.
Again by (\ref{tExpand}), we can set $\alpha_2 = x$.
For $m\geq 3$, we will now recursively define $\alpha_m\in L(mO)\cap \hat R[x,y]$. Writing
$\tilde{\alpha}_m = \frac{1}{t^m} + \frac{a_{-(m-1)}}{t^{m-1}} + \cdots + \frac{a_{-2}}{t^2} + \frac{a_{-1}}{t} + a_0 + a_1t+ \cdots$,
for some $a_\ell$, $\ell \geq -(m-1),$ in $\hat R$,
we now set $\alpha_m$ to be $\tilde{\alpha}_m-\sum_{\ell=2}^{m-1}a_{-\ell} \alpha_\ell-a_0$. 
Then since the $a_\ell$ are in $\hat R$, we have that $\alpha_m$ is in $\hat R[x,y]\cap L(mO),$ 
and by design
$$\alpha_m = \frac{1}{t^m}+\frac{r_m}{t}+0 + s_mt+\cdots,$$ 
for some $r_m, s_m \in \hat R$.


If $\alpha'_m$ were another function in $L(mO)$ whose expansion in $t$
was of the form $\frac{1}{t^m}+\frac{r'_m}{t}+0 + s'_mt+\cdots,$
then $\alpha_m-\alpha'_m\in L(O)$, so would be a constant. That forces $r_m=r'_m,$
whence $\alpha_m-\alpha'_m$ vanishes at the origin, so is $0$. 

In particular, $\alpha'_m=(-1)^m[-1]^*\alpha_m$ is of this form, so  $(-1)^m[-1]^*\alpha_m=\alpha_m$, hence
$\alpha_m$ is even when $m$ is even and odd when $m$ odd.

To compute weights, we can use that the weight of $t$ is $-1$ to evaluate $gr_\kappa(\alpha_m)$ for any $\kappa$ in $\Z_p^\times$ of infinite-order, and find that
\begin{align*}
gr_\kappa(\alpha_m)
&=\frac{1}{gr_\kappa(t)^m}+\frac{gr_\kappa(r_m)}{gr_\kappa(t)}+0 + gr_\kappa(s_m)gr_\kappa(t)+\cdots \\
&=\frac{\kappa^m}{t^m}+\frac{gr_\kappa(r_m)\kappa}{t}+0 + \frac{gr_\kappa(s_m)}{\kappa}t+\cdots.
\end{align*}
By the uniqueness of $\alpha_m$, we have that $\alpha_m=\kappa^{-m} gr_\kappa(\alpha_m)$. So $gr_\kappa(r_m)=\kappa^{m-1}r_m$ and $gr_\kappa(s_m)=\kappa^{m+1}s_m$.
Hence $\alpha_m$ is   of weight $m$, $r_m$ is of weight $m-1$, and $s_m$ is of weight $m+1$.

(b) This follows from (a) by taking $z_n = \alpha_{p^n}$, $H_n = -r_{p^n}$, $J_n = - s_{p^n}$, and noting that $z_n$ is odd since $p^n$ is.


(c) Since $z_n\in \hat R[x,y]$, it reduces modulo $p$ to a function $\bar{z}_n$ in $k[x,y]$,
which therefore has poles only at the origin on $ \mathcal E_p$, and an expansion there of the form
  $$t^{-p^n}-\gamma/t-\delta t+ ...,\gamma,\delta\in \hat R/p\hat R,$$
  where $\gamma=\bar{H}_n$ and $\delta=\bar{J}_n$, and
  a bar denotes reduction of elements in $\hat R$ modulo $p$.
  By the argument in (a), {\it mutatis mutandi}, $\bar{z}_n$
  is the unique such function in $k[x,y]$ with an expansion of this form. 
In particular, 
$$\bar{z}_{n}= (\bar{z}_{n-1})^p+\bar{H}_{n-1}^p\bar{z}_1,$$
hence $\bar{H}_{n}=\bar{H}_{n-1}^p\bar{H}_{1}$,
which gives recursively that $\bar{H}_{n} = \bar{H_1}^{1+p+\cdots +p^{n-1}}.$ As explained in Remark \ref{hasse}, $\bar{H}_1=\bar{H}$.  
Hence $H_n$ is invertible in $\hat{R}$ for all $n\geq 1$, and
it makes sense to define
$\zeta_n={H_n}^{-1}(t^{-p^n}-z_n)$ in $\frac{1}{t}+ t\hat{R}[[t]]$. 

(d) By (\ref{tExpand}), $W(t)^{-1}\in \hat R[[t]]$ and is $1\mod{t^4}$.
From (c), $-H_n^{-1}z_n\in \hat{R}[x,y]$, so $D(-H_n^{-1}z_n)$ is some polynomial $g(x,y)$ with coefficients in $\hat{R}$. Note that by the definition of $D$, the expansion at the origin of $g(x,y)$ is $-H_n^{-1}{{dz_n}\over {dt}}W(t)^{-1}$, which is 
$$ -H_n^{-1}
({{-p^n}\over{t^{p^n+1}}}+{{H_n}\over{t^2}}-J_n)\mod{t^2},$$
so 
$$g(x,y)=-1/t^2+\beta_n\mod{(p^n,t^2)},$$
where $\beta_n=J_n/H_n\in \hat{R}$.
Hence by (\ref{tExpand}), $g(x,y)=-x+\beta_n\mod{p^n}$. Since 
$D(t^{p^n})\equiv 0\mod{p^n}$, we get that $D\zeta_n\equiv -x+\beta_n\mod{p^n}$.

\end{proof}

\begin{lemma}\label{congruence} Let $g(t)\in \hat{R}((t))$ be a Laurent series such that $D(g(t))\equiv c_n~\mod{p^n}$ for
some $n\geq 1$ and some $c_n\in \hat{R}$. Then $c_n\equiv 0\mod{p^n}$.

Hence if $D(g(t))=c$ for some $c\in \hat{R}$, then $c=0$ and $g(t)$ is a constant.
\end{lemma}

\begin{proof} We can rewrite the condition $D(g(t))\equiv c_n~\mod{p^n}$ as ${{dg}\over{dt}}\equiv c_n
W(t)\mod{p^n}$. Since the coefficient of $t^{p^n-1}$ in ${{dg}\over{dt}}$
vanishes mod $p^n$, we have  $c_nw_{p^n-1}\equiv 0\mod{p^n}$. Note that the sum of the residues of $z_n \omega $ is 0, and since the only pole is at the origin, its residue there must vanish. The residue
 is $w_{p^n-1}-H_n,$ so $w_{p^n-1}=H_n$
is invertible mod $p$. Hence $c_n\equiv 0 \mod{p^n}$. (n.b.: That $w_{p-1}\equiv H_1\mod{p}$ is in \cite{Has}.)
\end{proof}

\medskip
\noindent
{\it Proof of Theorem \ref{theorem2}:} Applying Proposition \ref{betaprop} to $\zeta_n$ and $\zeta_{n+1}$ shows that $\beta_{n+1}\equiv \beta_n\mod{p^n}$, so we can set $\beta=\lim_{n\rightarrow\infty} \beta_n\in \hat{R}$.
Proposition \ref{betaprop}(d) shows that for every $m\geq 1$, $(-x(t)+\beta)W(t)$ is the derivative with respect to $t$
of a Laurent series in $t$ over $\hat{R}/p^m\hat R$, and hence
for every $n\geq 1$, the coefficient of $t^{pn-1}$ in the expansion of $(-x(t)+\beta)W(t)$ vanishes mod $pn$.
Therefore there is a Laurent series $\zeta(t) \in 1/t + \hat{R}[[t]]$ with the property that
$D(\zeta(t))=-x(t)+\beta$, and then by Lemma \ref{congruence} we can make $\zeta(t)$ unique by specifying that it is odd in $t$ (in fact, it is then given by the term-by-term limits of the $\zeta_n$). 
Hence $\zeta(t)$ is the unique choice for $\zeta_{\cal E/\hat{R}}(t)$.

\begin{center}\item\subsection{The construction of  the universal $p$-adic $\sigma$-function}\end{center}
To remove the polar term, we define $\tilde{\zeta}_{\cal E/\hat{R}}(t)\in \hat{R}[[t]]$ as
$$\tilde{\zeta}_{\cal E/\hat{R}}(t)=\zeta_{\cal E/\hat{R}}(t)-Dt/t=\zeta_{\cal E/\hat{R}}(t)-{1\over t}{{dt}\over{\omega}}=\zeta_{\cal E/\hat{R}}(t)-{1\over{tW(t)}}.$$
Let $\log(1+t)=\sum_{n\geq 1}(-1)^{n+1}{{t^n}\over n}$, and $\exp(t)=\sum_{n\geq 0} {{t^n}\over {n!}}$, so
we have in $\mathbb Q[[t]]$ that $\exp(\log(1+t))=1+t$, and $\log(\exp(t))=t$.
We let $\Lambda(t)$ denote the integral with respect to $t$ of $\tilde{\zeta}_{\cal E/\hat{R}}(t)W(t)$ in $\hat{R}[[t]]\otimes \mathbb Q$
that has no constant term, which will just write as 
\[\Lambda(t) =\int \tilde{\zeta}_{\cal E/\hat{R}}(t)\omega\]
which is even in $t$.

\begin{definition}\label{definition-of-sigma}
We define $\tilde{\sigma}(t)=\exp(\Lambda(t))$, as a power series in $t$
with coefficients in $\hat{R}\otimes \mathbb Q$.
We then set $\sigma(t)=t \tilde{\sigma}(t)$.
\end{definition}

Note that $\tilde{\sigma}(t)$ is an even power series in $t$ whose constant term is 1.
Hence $\sigma(t)$ is an odd power series in $t$
with lead term $t$, and by construction, $\zeta_{\cal E/\hat{R}}(t)=D(\sigma(t))/\sigma(t)$. 

Our goal is to show that $\sigma(t)$ (or equivalently $\tilde{\sigma}(t)$) has coefficients in $\hat{R}$, so that
$\sigma(t)$ will be the $\sigma_{\cal E/\hat{R}}(t)$ promised in Theorem \ref{theorem1}.
We will do this by using Hazewinkel's  functional equation lemma adapted to our situation (see \cite{Haz}, Chpt 1, Sect. 2; see also \cite{Ho} Lemma 2.4.)
We standardly let $\mathbb Z_{(p)}$ denote the ring of fractions $(\mathbb Z-p\mathbb Z)^{-1}\mathbb Z.$

\begin{lemma}[Functional Equation Lemma]\label{FunEqLem} 
Let $B$ be a $\mathbb{Z}_{(p)}$-algebra and $\alpha: B \rightarrow B$ be an
injective homomorphism such that for all $r\in B$, $\alpha(r)\equiv r^p\mod{p}$. Suppose $B$ is an integral domain and let $F$ be its field of fractions. Let $s$ be an indeterminate.
Extend $\alpha$ to $F$ and then to $F[[s]]$
by acting on the coefficients of power series.  Let  $a, b\in sF[[s]]$ be such that $a(s)-{1\over p}\alpha(a)(s^p)\in B[[s]]$ and $b(s)-{1\over p}\alpha(b)(s^p)\in B[[s]].$
Then $b^{-1}(a(s))\in B[[s]]$, where $b^{-1}$ denotes the power series in $F[[s]]$ such that $b^{-1}(b(s))=b(b^{-1}(s))=s$.  
\end{lemma}


\begin{proof}
This follows from the Functional Equation-Integrality Lemma in Section 1.2.2 of \cite{Haz}. In the notation therein, take $A=B,$ $K=F,$ $\sigma=\alpha,$ {\LARGE $a~$}$ = pB,$ $q=p,$ $s_1=p^{-1}$, and $s_2=s_3=\cdots = 0$.
Then if we let $g=b(s)-{1\over p}\alpha(b)(s^p)$
and $\overline{g}=a(s)-{1\over p}\alpha(a)(s^p),$
it can be shown that
$f_g=b(s)$ and $f_{\overline{g}}=a(s),$
in which case the result follows from \cite{Haz} I.2.2(ii).
\end{proof}


\begin{corollary}\label{FunEqCor} With $F$ and $\alpha$ as in Lemma \ref{FunEqLem}:

a) For any $a(s)\in sF[[s]]$ satisfying
$$a(s)-{1\over p}\alpha(a)(s^p)\in B[[s]],$$
we have that $\exp(a(s))$ has coefficients in $B$.

b) Suppose that $a \in s F[[s]]$ is such that $da/ds$ is in $B[[s]]$, so $a=\sum_{n=1}^\infty {{c_{n-1}}\over{n}}s^{n}$ for some $c_n\in B$. Then $a(s)-{1\over p}\alpha(a)(s^p)\in B[[s]],$ if and only 
if for all $n\geq 1$,
     $$c_{np-1}\equiv \alpha(c_{n-1})\mod{pn}.$$

c) Suppose $s' \in  s B[[s]]$ satisfies $s'\equiv s^p\mod{p}$.  If $a=\sum_{n=1}^\infty {{c_{n-1}}\over{n}}s^{n}$ for some $c_n\in B$, and if $a(s)-{1\over p}\alpha(a)(s')$ has coefficients in $B$, then 
$\exp{(a(s))}$ has coefficients in $B$.
\end{corollary}

\begin{proof} (a) This follows from Lemma \ref{FunEqLem} since $\alpha$ is trivial on the prime subfield 
$\mathbb Q$ of $F$, $\log{(1+s)}-{1\over p}\log{(1+s^p)}$ is in $B[[s]]$, and the inverse of $\log(1+s)$ is $\exp(s)-1$.

(b) This follows from comparing coefficients of $s^{np}$. 

(c) This follows from  (a) and (b)
since if in $B[[s]]$, $s'\equiv s^p\mod{p}$, then for all $n\geq 1$, $(s')^n \equiv s^{pn}\mod{pn}$.
\end{proof}

Our goal now is to apply part (c) of this corollary to  $\Lambda(t)$ when $B=\hat{R}$, $F=\hat{K}$, and $s=t$, which requires finding a ring endomorphism $\alpha$ of $\hat{R}$ that reduces to the Frobenius mod $p$, finding a suitable $s'$, and verifying the requisite functional
equation for $\Lambda$, which we now do in turn. 




Let $X$ be an indeterminate. Since $\hat{R}$ is $p$-complete, \cite{Ell} gives us a version of the  Weierstrass Preparation Theorem for $\hat{R}[X]$. We call a power series $f=\sum_{i\geq 0} a_iX^i \in \hat{R}[[X]]$ $p$-distinguished of order $n\geq 0$ if $a_i\equiv 0\mod{p}$ for $i<n$ and $a_n$ is a unit in $\hat{R}$.
We call a monic polynomial $P\in \hat{R}[X]$ of degree $n$ a $p$-Weierstrass polynomial if $P\equiv X^n\mod{p}$.

\begin{lemma}\label{weierstrass}  (a) For any $g \in \hat{R}[[X]]$ that is $p$-distinguished of order $n\geq 0$, there exists a unique $p$-Weierstrass polynomial $P \in \hat{R}[X]$ of degree $n$ and a unique unit $U\in 
\hat{R}[[X]]$ such that $g = UP$.

(b) If in (a) we have $g\in\hat{R}[X]$, then in the factorization $g=UP,$ $U \in  \hat{R}[X]$.
\end{lemma}

\begin{proof}
Part (a) comes from \cite{Ell} Theorem 1.3, and (b) is from \cite{Ell} Lemma 3.5.
\end{proof}

For any positive integer $n$, let $E[n]$ denote the $n$-torsion of $E$ in an algebraic closure $\overline{\hat K}$ of $\hat{K}$.

\begin{proposition} Let $C\subset E[p]$ be the subset consisting of $O$ and the points of $E[p]$ whose $x$-coordinates are not integral over $\hat{R}$.
Then $C$ is a subgroup of order $p$ defined over $\hat{K}$, called the canonical subgroup of $E$.
\end{proposition}

\begin{proof} For any non-zero integer $n$ we let $\psi_n$ denote the $n^{th}$-division polynomial for $E$, which is characterized by its divisor being $\sum_{u\in E[n]}u-n^2O$ and by $t^{n^2-1}\psi_n|_{t=0}=n.$ Is is well-known (see \cite{Was} chapter 3 for (i) and (iii) and \cite{Cas} Theorem 1 for (ii)) that: 

i) $\psi_n$ is in $\hat R[x]$ for $n$ odd, and in $2y\hat R[x]$ for $n$ even, and that with our weights on $\hat R$, $x$ and $y$, each term of $\psi_n$ has weight $n^2-1$.

ii) $D(\psi_p)\equiv 0\mod{p}$.


iii) $\varphi_n:=([n]^*x)\psi_n^2$ is in $\hat R[x]$ and is monic of degree $n^2$.


In particular, since $p$ is odd, by (i) we have
$$\psi_p(x)=px^{(p^2-1)/2}+\sum_{n=0}^{(p^2-3)/2} \ell_n x^n,$$
for some $\ell_n$ an integer polynomial in $A_4$ and $A_6$ which is   of weight $p^2-1-2n$. By (ii), $\ell_n$  is a multiple of $p$ if $n$ is not a multiple of $p$.

Most important for us is $\ell_{p(p-1)/2},$ which is of weight $p-1$. In \cite{Der} it is shown to be congruent to $H$ mod $p$, so
%
%
%
is invertible in $\hat{R}$. 
Hence $g(x):=x^{(p^2-1)/2}\psi_p(1/x)$ is a polynomial whose lowest non-zero term modulo $p$ is $\ell_{p(p-1)/2}x^{(p-1)/2}$. So applying Lemma \ref{weierstrass} when $X=x$ to $g$ gives
that over $\hat{R}$, $g(x)=U(x)P(x)$ where $P(x)$ is a $p$-Weierstrass polynomial of degree $(p-1)/2$ (so in particular is monic), and $U(x)$ is of degree at most $(p^2-p)/2$. If we let $\tilde{\ell}_0$ denote the
constant term of $U(x)$, it is congruent to 
$\ell_{p(p-1)/2}$ modulo $p,$ so is a unit in $\hat{R}$. Setting $\pi(x)=\tilde{\ell}_0 x^{(p-1)/2}P(1/x)$,
$\xi(x)=\tilde{\ell}_0^{-1}x^{(p^2-p)/2}U(1/x),$ which are in $\hat{R}[x]$, we get that
$\psi_p(x)=x^{(p^2-1)/2}g(1/x)=\pi(x)\xi(x)$. Since by design $\xi(x)$ is monic, we get that
$\pi(x)$ is of the form
\begin{equation}\label{piFactor}\pi(x)=px^{(p-1)/2}+\tilde{\ell}_{{(p-3)}/2}x^{{p-3}\over 2}+ ....+\tilde{\ell_0},
\end{equation}
for some  $\tilde{\ell_i}\in \hat{R},$ $1\leq i\leq (p-3)/2.$
As a result, the $p^2-p$ points of $E[p]$ which are in the divisor of zeroes of $\xi(x)$ have $x$-coordinates which are integral over $\hat{R}.$ On the other hand $p$ is prime in $\hat{R}$, and (\ref{piFactor}) shows
that by Eisenstein's criterion, $P(x)$ --- and hence $\pi(x)$ --- is irreducible over $\hat{R}$. Therefore the $p-1$ points $\{P_1,...,P_{p-1}\}$ of $E[p]$ in the divisor of zeroes of $\pi(x)$ have $x$-coordinates which are not integral over $\hat{R}$. We set $C=\{O,P_1,...,P_{p-1}\}$. 

It remains to be shown that $C$ is a subgroup of $E[p]$. We can show it is a cyclic subgroup of order $p$ by verifying that for every integer $i$ prime to $p$, we have $[i]C\subseteq C$, where $[i]$ denotes the multiplication-by-$i$ map on $E$ . Note that $[i]$ induces an automorphism $[i]_p$ of $E[p]$, whose inverse is given by $[j]_p$ for any integer $j$ such that $ji\equiv 1\mod{p}.$ Note $C$ is closed under $[i]_p$ if and only if its complement $\tilde{C}=E[p]-C$ is closed under $[i]_p$, i.e., $\tilde{C}$ is closed under $[j]_p^{-1}.$ Note also that
$$[j]_p^{-1}\tilde{C}=\{u\in E[jp]|[j]u\in \tilde{C}\}\cap E[p],$$
so it suffices to show that for every $u\in E[jp]$ such that $[j]u\in \tilde{C}$, that the $x$-coordinate of $u$ is integral over $\hat{R}$.
This however follows because 
these $x(u)$ are precisely the zeroes of $\xi([j]^*x)=\xi(\varphi_j/\psi_j^2)$, which since $j$ is prime to $p$ are the roots of $\xi(\varphi_j/\psi_j^2)(\psi_j)^{p^2-p}$, which by (iii) is a monic polynomial over $\hat{R}$.

Finally, since $C$ consists of the origin and the divisor of zeroes of $\pi(x)$, it is defined over $\hat{K}.$
\end{proof}

Using the proposition, we now let $E'=E/C$, and
and let $\phi: E\rightarrow E'$ be the induced isogeny over $\hat{K}$.
We note that there is not a unique Weierstrass model for $E'$, but for each non-zero $\gamma\in \hat{K}$ there is a unique 
Weierstrass model  $$E'_\gamma: {y_\gamma}^2=f_\gamma(x_\gamma)={x_\gamma}^3+A_{4,\gamma}'x_\gamma+A_{6,\gamma}',$$
with $\omega_\gamma=dx_\gamma/2y_\gamma,$ determined by the condition that $\phi^*\omega_\gamma/\omega = \gamma $. We will use $\phi$ to identify the function field $\hat{K}(E')$ with
a subfield of $\hat{K}(E)$.

A main technical result of the paper is the following, which says that with care we can find models for $E'$ over $\hat{R},$ which define other WOGECs (n.b. \cite{Co1} Example 2.1.6), and which will provide us with the map $\alpha$ needed in our applications of the functional equation lemma. 

\begin{proposition}\label{moddingout} a) The model  $E'_p$ is of the form
\begin{equation}\label{Ep'Model}{y_p}^2=f_p(x_p)={x_p}^3+A_{4,p}'x_p+A_{6,p}',\end{equation}
where $A_{4,p}'$ and $A_{6,p}'$ are   in $\hat{R}$ of weights $4$ and $6$ respectively, 
and reduce respectively to $A_4^p/H^4$ and $A_6^p/H^6$ mod $p$. 
In addition, if $t_p=-x_p/y_p$, then $t_p\equiv H t^p\mod{p}$.

b) The model  $E'_{p/H}$ is of the form
$${y_{p/H}}^2={x_{p/H}}^3+A_{4,p/H}'x_{p/H}+A_{6,p/H}',$$
where $A_{4,p/H}'$ and $A_{6,p/H}'$ are   in $\hat{R}$ of weights $4p$ and $6p$ respectively, 
and reduce respectively to $A_4^p$ and $A_6^p$ mod $p$. 

c) Let $t_{p/H}=-x_{p/H}/y_{p/H}$. Then $t_{p/H}$ is in $\hat{R}[[t]],$ is odd in $t$,
and there is a power series $v(t)\in 1+pt\hat{R}[[t]]$ such that
$$t_{p/H}={1\over H}(t^p\pi(x))v(t),$$
where $\pi(x)$ is as in (\ref{piFactor}).
\end{proposition}

\begin{proof} a) Let (\ref{Ep'Model}) be the model for $E'_p$. We now want to verify the claims about 
$A_{4,p}', A_{6,p}'$ and $t_p$. By way of notation, for any point $u$ of $E,$ let $\tau_u$ denote the translation-by-$u$ map on $E$, and 
for any function $g$ in the function field $\hat{K}(E)$ of $E$ regular on $C$, let $N(g)=\prod_{u\in C}(\tau_u)^*(g)$ be the norm and let
\begin{equation}\label{N0Def}
N_0(g)=\prod_{u \in C-O} g(u)=(N(g)/g)(O).
\end{equation}
It follows from (\ref{piFactor}) that $N_0(x)=\tilde{\ell_0}^2/p^2$.

Now let $r_i, r_i'$, $1\leq i\leq 3$ be respectively roots of $f$ and $f_p'$ in an algebraic closure $\overline{\hat K}$ of $\hat K$,
so $e_i=(r_i,0)$ and $e_i'=(r_i',0)$ are non-trivial 2-torsion points on $E$ and $E'$ respectively.
Let $O_{E'}$ denote the origin on $E'.$ Then $C=\phi^{-1}(O_{E'}),$ and
since $p$ is odd, reordering the $e_i'$ if necessary, we  can assume $\phi^{-1}(e_i')=\tau_{e_i}(C).$
Then comparing divisors, there are constants $c_i$ in  $\overline{\hat{K}}$ 
such that 
$ x_p-r_i'=c_i^2N(x-r_i),$ and $y_p=\pm c_1c_2c_3N(y),$  so $x_p={1\over 3}\sum_{i=1}^3c_i^2N(x-r_i).$

The expansion of $x_p-r'_i$ in terms of $t$ has a lead term independent of $i$, which from (\ref{tExpand}) and (\ref{N0Def}) we see is
$c_i^2N_0(x-r_i)/t^2$. There is therefore a constant $c\in \overline{\hat{K}}$ such that for all $i$,
$c_i^2=c^2/N_0(x-r_i)$. 
Then $x_p$ has a lead term $c^2/t^2$, so by (\ref{Ep'Model}) and replacing $c$ by $-c$ if necessary, we can take
$y_p$ to have a lead term $-c^3/t^3$, and so if $t_p=-x_p/y_p$, then $t_p$ has lead term $t/c$.
Therefore $\phi^*(\omega_p)=\phi^*(dx_p/2y_p)=\omega/c$.
Hence by design,  $c=1/p$, so
\begin{equation}\label{xp}
x_p={1\over 3}\sum_{i=1}^3 N(x-r_i)/p^2N_0(x-r_i).\end{equation}
Note that $y_p$ has a lead term of $-1/p^3t^3$, and by the above, is a constant times $N(y)$ --- whose lead term by definition is $N_0(y)$ times the lead term $-1/t^3$ of $y$. So we have
\begin{equation}\label{yp}
y_p=N(y)/p^3N_0(y).\end{equation}
We now set out to calculate $N(x-r_i).$

We claim that for any $u\in C-O$, 
\begin{equation}\label{SymProdxri}
(\tau_u^*(x)-r_i)(\tau^*_{-u}(x)-r_i)=\bigg({{(x-r_i)(\tau_{e_i}(x)-x(u))}\over{x-x(u)}}\bigg)^2.
\end{equation}
Indeed the divisor of both sides of (\ref{SymProdxri}) 
is $2(u+e_i)+2(-u+e_i)-2u-2(-u),$
%
and both sides of (\ref{SymProdxri}) are $(x(u)-r_i)^2$ at the origin, so are the same.
An exercise with the group law on $E/\hat{K}$ shows
\begin{equation}\label{(x-ri)tau}
(x-r_i)\tau_{e_i}(x)=-(x+r_i)(x-r_i)+y^2/(x-r_i)=xr_i+r_jr_k+r_i^2,
\end{equation} where $\{i,j,k\}=\{1,2,3\}$,
which lies in $\hat{R}[r_1,r_2,r_3].$

%
%

Now taking the product of (\ref{SymProdxri})
over the cosets of the non-identity elements of $C$ under the action of $[\pm 1]$ and then
multiplying by $x-r_i$ gives that 
\begin{equation}\label{N(x-ri)Form1}
N(x-r_i)=(x-r_i)^p\bigg({{\pi(\tau_{e_i}(x))}\over{\pi(x)}}\bigg)^2.
\end{equation}
Since $\pi(\tau_{e_i}(x))\equiv \pi(x) \equiv \tilde{\ell_0}\mod{p},$ (\ref{N(x-ri)Form1})
gives that $N(x-r_i)$ reduces to $(x-r_i)^p$ mod $p$.

From (\ref{N0Def}) and (\ref{N(x-ri)Form1}) 
we also get that $p^2N_0(x-r_i)=p^2\left.{{N(x-r_i)}\over{x-r_i}}\right\vert_O=\pi(r_i)^2$, where $\pi(r_i)$ 
reduces to $\tilde{\ell_0}$ mod $p$, and is a unit in $\hat{R}[r_i]$, since $\pi(r_1)\pi(r_2)\pi(r_3)$ is in $\hat{R} $ and reduces to $\tilde{\ell_0}^3$ mod $p$.
If we rewrite
\begin{equation}\label{N(x-ri)Form2}
N(x-r_i)=(x-r_i)S_i^2(x)/\pi(x)^2,
\end{equation}
where $S_i(x)=(x-r_i)^{(p-1)/2}\pi(\tau_{e_i}(x)),$
then (\ref{(x-ri)tau}) shows that $S_i(x)$ is in $\hat{R}[r_1,r_2,r_3][x]$, is of degree $(p-1)/2$, and reduces to $(x-r_i)^{(p-1)/2}\tilde{\ell_0}$ mod $p$. 
Putting these together,  (\ref{xp}) gives that 
$$x_p ={1\over 3}\sum_{i=1}^3 {{(x-r_i)S_i^2(x)}\over{\pi(r_i)^2\pi(x)^2}}={{S(x)}\over{\pi(x)^2}},$$
for some polynomial $S(x)$ which by symmetry is in $\hat{R}[x]$, is of degree $p$, and reduces to $ x^p$ mod $p$. Hence $x_p \equiv x^p/\tilde{\ell}_0^2\mod{p}$.

Likewise, taking a product of (\ref{N(x-ri)Form1}) over $i$, using that $p^6N_0(y^2)=\pi(r_1)^2\pi(r_2)^2\pi(r_3)^2,$ a unit in $\hat{R}$, we get from (\ref{yp}) and (\ref{N(x-ri)Form2}) that 
$$y_p^2=(N(y)/p^3N_0(y))^2=\prod_{i=1}^3{(x-r_i){S_i(x)^2}\over{\pi(r_i)^2\pi(x)^2}}=    \bigg({{yM(x)}\over{\pi(x)^3}}\bigg)^2,$$
where $M(x)=\prod_{i=1}^3 {{S_i(x)}\over {\pi(r_i)}} \in \hat{R}[x]$ has degree $(3p-3)/2$, and  $M(x)\equiv f(x)^{(p-1)/2} \mod{p}.$ 
We've seen that $y_p/y$ at the origin is $1/p^3$, and ${{M(x)}\over{\pi(x)^3}}$ at the origin is the lead coefficient of $M$ --- which is $1\mod{p}$ --- divided by $p^3.$
Therefore $y_p={{yM(x)}\over{\pi(x)^3}}$, and so
$y_p\equiv y^p/\tilde{\ell}_0^3\mod{p}$. Hence $t_p\equiv -x_p/y_p\equiv\tilde{\ell}_0 t^p\equiv H t^p\mod{p}$. 

Using these expressions for $x_p$ and $y_p$ and multiplying (\ref{Ep'Model}) by $\pi(x)^6$ shows that
$$A_{4,p}' S(x) \pi(x)^4+A_{6,p}' \pi(x)^6 \in \hat{R}[x].$$
A priori we only know that $A'_{4,p}$ and $A'_{6,p}$ lie in $\hat{K}$, but
since the constant term of $\pi(x)$ is a unit in $\hat{R}$, Gauss's lemma gives that
\begin{equation}\label{A4'S+A6'pi}
A_{4,p}' S(x)+A_{6,p}' \pi(x)^2 \in \hat{R}[x].
\end{equation}
The coefficient of $x^p$ in (\ref{A4'S+A6'pi}) is a unit in $\hat{R}$ times $A'_{4,p}$ so we get $A'_{4,p}\in\hat{R}$.
We conclude that $A_{6,p}' \pi(x)^2 \in \hat{R}[x],$ and as above, that $A'_{6,p}\in\hat{R}$.

Hence from (\ref{Ep'Model}) we get
$$y_p^2\equiv y^{2p}/H^6\equiv (x^3+A_4x+A_6)^p/H^6\equiv x_p^3+(A_4^p/H^4)x_p+A_6^p/H^6\mod{p}.$$
Therefore $A'_4\equiv A_4^p/H^4\mod{p}$, and $A'_6\equiv A_6^p/H^6\mod{p}.$

We now need to show that the construction gives that $A_{4,p}'$ and $A_{6,p}'$ have the desired weights. There are two key points. 

The first is that the factorization in the Weierstrass
Preparation Theorem uniquely gives us
$\psi_p(x)=\pi(x)\xi(x)$, with $\pi(x)$ of degree $(p-1)/2$ with lead coefficient $p$.
Indeed $P(x)$ was unique, and $\pi(x)$ is the unique constant multiple of $x^{(p-1)/2}P(1/x)$ which has lead coefficient $p$. By this uniqueness, for every $\kappa\in \mathbb Z_p^\times$,
$\psi_p(x)=\kappa^{-(p^2-1)}gr_\kappa(\psi_p(x)),$ 
and we therefore get $\pi(x)=\kappa^{-(p-1)}gr_\kappa(\pi(x)),$ so $\pi(x)$ has weight $p-1$. Hence 
each $\tilde{\ell_n}$ has weight $p-1-2n.$ The second point is that since
$f(x)$ has weight 6,
we can assign each $r_i$ a weight of 2 and turn $\hat{R_g}[r_1,r_2,r_3]$ into a graded ring which contains $\hat{R_g}$ as a graded subring. This gives us that
the $\pi(r_i)$ have weight $p-1$, so the weight of $N_0(x-r_i)$ is $p-1$, and from (\ref{(x-ri)tau}) that $\tau_{e_i}(x)$ has weight 2. Then (\ref{N(x-ri)Form1}) gives that $N(x-r_i)$ has weight $2p,$ and (\ref{N(x-ri)Form2}) says that
$S_i(x)$ has weight $2p-2$. Hence $x_p$ has weight $2$ and $y_p$ has weight $3.$ It follows that the expression in (\ref{A4'S+A6'pi}) has weight $2p+4,$ and its coefficient of $x^p$ is $A'_4$, which hence has weight {4}. Therefore $A'_6\pi(x)^2$ has weight $2p+4$, so $A'_6$ has weight $6$.

b) This follows from the effects of changing Weierstrass models, and that $H$ has weight $p-1.$
 
c) Since $M(x)$ has a lead coefficient that is a unit in $\hat{R}$, we have  \begin{equation}\label{t'formula}
t_{p/H}=-S(x)\pi(x)/H M(x)y=-(t^{2p}S(x))(t^{p}\pi(x))/H(t^{3p-3}M(x))(t^3 y)
\end{equation}
 is a power series in $\hat{R}[[t]]$ divided by an invertible power series in $\hat{R}[[t]]$, so lies in $\hat{R}[[t]]$. Note (\ref{t'formula}) expresses $t_{p/H}$
 as an odd function on $E$, 
 so  
 $t_{p/H}$ is odd in $t$. Set
 $v(t)=-t^{2p}S(x)/(t^{3p-3}M(x))(t^3 y)\in \hat{R}[[t]],$
so $t_{p/H}= {1\over H}(t^p\pi(x))v(t)$.
 
 The lead term of $t_{p/H}$ is $pt/H$ because $\phi^*(\omega_{p/H})={p\over H}\omega$, and 
 hence the lead term of $v(t)$ is $1$ since the lead term of $t^p\pi(x)$ is $pt$. Finally $v(t)$ mod $p$ is
 $$-t^{2p}x^p/t^{3p-3}y^{p-1}t^3y=(-x/ty)^p=1.$$
 
\end{proof}

\begin{remark} In Appendix I of \cite{MT} they define a division polynomial for any isogeny of elliptic curves normalized by a choice of invariant differentials on the curves. Using this definition, that $v(t)$ in (c) above has constant term 1 implies that $\pi(x)/H$ is the division polynomial of $\phi$ given the choices of $\omega$ and $\omega_{p/H}$ for invariant differentials on $E$ and $E'.$
\end{remark}

\begin{definition}\label{DefAlpha} Let $\alpha$ be the homomorphism from $R$ to $\hat{R}$ that sends $(A_4,A_6)\rightarrow (A_{4,p/H}', A_{6,p/H}')$, which 
Proposition \ref{moddingout} shows has the property that $\alpha(r)\equiv r^p\mod{p}$ for any $r\in R$. Hence $\alpha(H)$ is $H^p$ mod $p$, so is invertible in $\hat{R}$. Therefore $\alpha$ extends uniquely to $R_H$ and thence continuously to $\hat{R}$, where it reduces to the Frobenius mod $p$. We also denote this extension to $\hat{R}$ by $\alpha$. 
(At the end of the section we will also consider the analogous weight-preserving endomorphism $\alpha_0:\hat{R}\rightarrow \hat{R}$ determined by $(A_4,A_6)\rightarrow (A_{4,p}', A_{6,p}').$)
\end{definition}

From now on we take $E'_{p/H}$ as the defining model for $\cal E'$. We correspondingly define 
$$H'=H(A_{4,p/H}',A'_{6,p/H})\equiv H^p\mod{p},$$ so $H'$ is invertible in $\hat{R},$ and $\cal E'$ again defines a WOGEC. 
Likewise we set $\omega'=\omega_{p/H}$ and $t'=t_{p/H}$.
Let $D'$ be the derivation on $\hat{K}(E')$ defined by $D'(g)=dg/d\omega'$, which is to say, the derivation determined by $D'(x_{p/H})=2y_{p/H}.$ Since $\phi^*(\omega')=(p/H)\omega,$
for any $g\in \hat{K}(E')$, we have $D'(g)=(H/p)D(g)$. We showed in \S2.1 that $D$ has a unique extension to $\hat R[[t]]$, and likewise $D'$ has a unique extension to $\hat R[[t']]$, which implies that for any
Laurent series $a(t) \in \hat{R}((t'))$, we also have $D'(a)=(H/p)D(a)$.

Since $\cal E'$ is an $\alpha$-specialization of $\mathcal E$, using Corollary \ref{sigma-specialization}
we have a Laurent series $\zeta_{\cal E'/\hat{R}}(t')=\alpha(\zeta_{\cal E/\hat{R}})(t')$ such that 
$$D'(\zeta_{\cal E'/\hat{R}}(t'))=-x_{p/H}(t')+\alpha(\beta),$$
where $\alpha(\beta)$ is in $\hat{R}$.
Furthermore from (4), we get that $D'(t')=1/\alpha(W)(t').$

In parallel to the definitions at the beginning of this Section, we now  set $\tilde{\zeta}_{\cal E'/\hat{R}}(t')=\zeta_{\cal E'/\hat{R}}(t')-D't'/t'$, 
which since $D't'/t'=1/t'\alpha(W)(t')$,  is the same thing as $\alpha (\tilde{\zeta}_{\cal E/\hat{R}})(t').$

So we get from part (c) of Proposition \ref{moddingout} that:
\begin{corollary}\label{integrality}
We have $\tilde{\zeta}_{\cal E'/\hat{R}}(t')\in \hat{R}[[t]].$
\end{corollary}

To complete the proof of Theorem 1 we need to verify that
with our definitions of $\alpha$ and setting $s'=t'$, the coefficients of $\Lambda(t)$ also meet the requisite criteria
in part (c) of the Corollary to the Functional Equation Lemma.
For this we need two lemmas,  the first whose proof follows readily from the group law on $E/\hat{K}$,
and the second of which is due to V\'elu \cite{Ve} (see also \cite{Elk}).

We note that there is a unique way to extend $D$ to a $\overline{\hat{K}}$ derivation on $\overline{\hat{K}}(E)$, and we will also denote that extension by $D.$

\begin{lemma}\label{grouplaw} For any point $u\in E$ other than $O$,
$$D ({{D(x-x(u))}\over{x-x(u)}})= 2x-(\tau_u^*x+\tau_{-u}^*x) . $$
\end{lemma}\qed

\begin{lemma}[V\'elu]\label{velu} 
For $g\in \hat{K}(E),$ let $T(g)=\sum_{u\in C}\tau_u(g)$ be the trace, and $T'(g)=\displaystyle \sum_{u\in C-O}g(u)$.
Then the model $E_1'$ for $E'$ is 
$$y_1^2=x_1^3+A'_{4,1}x_1+A'_{6,1},$$
for some $A'_{4,1}, A'_{6,1}\in \hat{K}$, where $x_1=T(x)-T'(x),$ $y_1=T(y)-T'(y).$
\end{lemma}\qed



We can now prove:

\begin{proposition}\label{zetas} Keeping the above notation:

a) $$\zeta_{\cal E'/\hat{R}}(t')=H \zeta_{\cal E/\hat{R}}(t)+{H\over p}{{D\pi(x(t))}\over{\pi(x(t))}}.$$

 b)  \[ \Lambda(t)-{1\over p}\alpha(\Lambda)(t')=\int\tilde{\zeta}_{\cal E/\hat{R}}(t)\omega-{1\over p}\int\tilde{\zeta}_{\cal E'/\hat{R}}(t')\omega'\in \hat{R}[[t]],\]
 where the integrals are taken to have vanishing constant terms.
\end{proposition}


\begin{proof}

Let $ \epsilon=\zeta_{\cal E'/\hat{R}}(t')-H \zeta_{\cal E/\hat{R}}(t)-{H\over p}{{D\pi(x(t))}\over{\pi(x(t))}},$
which is \emph{a priori} in $\hat{K}((t)).$
Our goal is to show that $ \epsilon=0$: we will do this in stages.

We first claim that $ \epsilon\in\hat{R}[[t]].$
By Corollary \ref{integrality} it
suffices to show that 
$$D't'/t'-{H\over p}{{D\pi(x(t))}\over{\pi(x(t))}}-HDt/t \in \hat{R}[[t]].$$
But since $D't'/t'=(H/p)Dt'/t'$,  
Proposition \ref{moddingout} (c) 
shows that this expression can be written as $HDv(t)/pv(t),$ for some $v(t)\in 1+pt\hat{R}[[t]],$ which gives us our claim.

It follows that $\eta=D( \epsilon)\in \hat{R}[[t]].$ We will now show that as an element of $\hat{K}(E)$, $\eta\in \hat{K}.$

Working first in $\overline{\hat K}(E),$ we compute using Lemma \ref{grouplaw} and Lemma \ref{velu}
that:

$$D\bigg( {H\over p}{{D\pi(x(t))}\over{\pi(x(t))}}\bigg)={H\over p}\sum_{u\in (C-\{O\})/\pm 1}
D \bigg({{D(x-x(u))}\over{x-x(u)}}\bigg)$$
$$={H\over p}\sum_{u\in (C-\{O\})/\pm 1} (2x-(\tau_u^*x+\tau_{-u}^*x))$$
$$=Hx-{H\over p}\sum_{u\in C} \tau_u^*x=Hx-{H\over p}(x_1+T'(x)).$$
Now using Theorem 3 and working in  $\hat{K}((t))$ we have: 
$$\eta=D\bigg(\zeta_{\cal E'/\hat{R}}(t')-H \zeta_{\cal E/\hat{R}}(t)-{H\over p}{{D\pi(x(t))}\over{\pi(x(t))}}\bigg)$$
$$={p\over H}D'(\zeta_{\cal E'/\hat{R}}(t'))-HD(\zeta_{\cal E/\hat{R}}(t))-Hx+{H\over p}(x_1+T'(x))$$
$$={p\over H}(-x_{p/H}+\alpha(\beta))-H(-x+\beta)-Hx+{H\over p}(x_1+T'(x))$$ 
$$={p\over H}(-x_{p/H}+\alpha(\beta))-H\beta+{H\over p}(x_1+T'(x)).$$
Since $x_{p/H}={{H^2}\over {p^2}}x_1$, we get that $\eta={p\over H}\alpha(\beta)-H\beta+
{H\over p}T'(x)\in \hat{K}$ as desired.


Since $ \epsilon\in \hat{R}[[t]],$ we actually have\footnote{We also have that ${{T'(x)}\over p}={{-2 \tilde{\ell}_{{p-3}\over 2}}\over p^2} \in \hat{R}$, which is not hard to see directly. For example, that $\ell_{{p^2-3}\over 2}=0$ and ${d\over {dx}}\xi_\phi(x)\equiv 0\mod{p}$ implies that $\tilde{\ell}_{{p-3}\over 2}\equiv 0\mod{p^2}.$} that $\eta\in \hat{R}$. It follows then from Lemma \ref{congruence} that $\eta=0.$ Hence $ \epsilon$ is constant, i.e. is in $\hat{R}.$ Since by Proposition \ref{moddingout} (c) it is also an odd power series in $t,$ we have that $ \epsilon=0$.

To prove (b), note that from (a) we have
$${H\over p}\tilde{\zeta}_{\cal E/\hat{R}}(t)-{1\over p}\tilde{\zeta}_{\cal E'/\hat{R}}(t')= 
{H\over p}\zeta_{\cal E/\hat{R}}(t)-{1\over p}\zeta_{\cal E'/\hat{R}}(t')-{H\over p}{{Dt}\over t}+{1\over p}{{D't'}\over {t'}}=$$ 
$${{-H}\over {p^2}}{{D\pi(x(t))}\over{\pi(x(t))}}-{H\over p}{{Dt}\over t}+{H\over {p^2}}{{Dt'}\over {t'}}={{-H}\over {p^2}}{{D(t^p\pi(x(t))/H t')}\over{t^p\pi(x(t))/H t'}}=
{{-H}\over {p^2}}{{D(1/v(t))}\over{1/v(t)}}
={{H}\over {p^2}}{{D(v(t))}\over{v(t)}}.$$
Multiplying by $\omega'(t')={p\over H}\omega(t)$ and integrating gives 
$$\Lambda(t)-{1\over p}\alpha(\Lambda)(t')=
{{\log(v(t))}\over p}.$$
The proof is completed by the observation that 
since $v(t)\equiv 1\mod{p}$, we have
${{\log(v(t))}\over p}\in \hat{R}[[t]]$.\qed


\medskip
\noindent
{\it Proof of Theorem \ref{theorem1}}. 
Write $\tilde{\zeta}_{\cal E/\hat{R}}(t)=\sum_{n\geq 1}c_nt^n,$ 
so $\Lambda(t)=\sum_{n\geq 2}{{c_{n-1}}\over n}t^n.$
Then Proposition \ref{zetas} (b) 
says we can apply part (c) of Corollary \ref{FunEqCor} to the Functional Equation Lemma 
with $F=\hat{K}$, $s=t$, and $s'=t'$ to deduce
that $\tilde{\sigma}(t)$ --- and hence $\sigma(t)$ ---  has coefficients in $\hat{R}$. Therefore we can take $\sigma_{\cal E/\hat{R}}(t)$ to be $\sigma(t)$. 
As for uniqueness, it follows from the uniqueness of $\zeta_{E/\hat{R}}$ that any two possible candidates for $\sigma_{E/\hat{R}}(t)$ have as a ratio a
unit power series $e(t)\in\hat{R}[[t]]$ with lead term 1 such that $De/e=0$. Hence such an $e$ is 
a constant, so must be $1$.
\end{proof}

\begin{remark}\label{modular}
We can gain some insight into our construction by considering various quantities as $p$-adic modular forms.
Let $\cal M$ denote the ring of level-one $p$-adic modular forms (with growth condition ``$r=1$'' \cite{K}).
One standardly embeds $R$ into $\cal M$ 
by setting  $i(A_4,A_6)=(-\frac{\mathbb E_4}{48}, \frac{\mathbb E_6}{864}),$ where $\mathbb E_{2n}$ is the normalized Eisenstein series of weight $2n$. Since $H$ gives the Hasse invariant for an elliptic curve in the form (1) over a field of characteristic $p$, $i(H)\equiv\mathbb E_{p-1}\mod{p}$,
which is invertible in $\cal M$ since $\cal M$ is $p$-complete. 
Hence we can extend $i$ to an embedding of $R_H$, which then
extends to an embedding of $\hat{R}$ into $\cal M$,
using again that $\cal M$ is $p$-complete.

By considering their construction of the $p$-adic sigma function applied to the Tate curve, Mazur and Tate computed the $q$-expansion of $i(\beta)$ and showed
\begin{equation}\label{iBeta}
i(\beta)={1\over {12}}\mathbb E_2
\end{equation} 
(n.b. the sign correction in \cite{MST}).

Now let $\alpha$ and $\alpha_0$ be as in Definition \ref{DefAlpha}.

 

Recall (see e.g. \cite{G}, II.2) that the $\Frob$ operator on $\cal M$ is obtained by first applying the $V$ operator 
which maps modular forms of level 1 to forms on $\Gamma_0(p)$ (by replacing $q$ by $q^p$ in their $q$-expansions) and then embedding the latter into forms of level 1.
It follows from the results in \S3 of \cite{K} 
that $\alpha_0$ is a lift of $\Frob$ to $\hat{R}$, that is, $\Frob\circ\, i=i\circ \alpha_0$. Note that our definition of $\beta$ as a limit of $\beta_n=J_n/H_n$ (see Proposition \ref{betaprop}) shows immediately that $i(\beta)$ is a $p$-adic modular form of  weight 2, and hence that $\alpha(\beta)=H^2 \alpha_0(\beta)$.  

In the course of the proof of Proposition \ref{zetas} we showed that
$$p\alpha_0(\beta)={p\over {H^2}}\alpha(\beta)=\beta-{1\over{p}}T'(x),$$
where $T'(x)$ is defined in Lemma \ref{velu}, and is $-2\tilde{\ell}_{{p-3}\over 2}/p$ in the notation of (\ref{piFactor}).

Applying the embedding $i$ gives
$$p \Frob(i(\beta))=i(\beta)-{{i(T'(x))}\over {p}},$$
which in light of (\ref{iBeta}), is the statement that
$$ i(T'(x)/p)={{1-p}\over {12}}\mathbb E_2^*,$$
where $\mathbb E_2^*:=(\mathbb E_2-p\Frob(\mathbb E_2))/(1-p)$ is the weight 2 $p$-adic Eisenstein series described in \cite{Se},
whose $q$-expansion is $1-{{24}\over{1-p}}\sum_{n\geq 1}\sigma^*(n)q^n$, where $\sigma^*(n)$ is the sum of the divisors of $n$
prime to $p$.
\end{remark}

\section{Universal equivalent formulations and specializations.}

Recall that if $A$ be a complete DVR of residue characteristic $p>3$ and $E/A$ an elliptic curve with good ordinary or multiplicative reduction over $A$, given by a Weierstrass model 
\begin{equation}\label{WeierModel}
y^2=x^3+a_4 x+ a_6,~t=-x/y, ~\omega=dx/2y,
\end{equation}
then Mazur and Tate attached a $p$-adic sigma function $\sigma_{E/A}$ to this model, which they proved is the unique power series in $A[[t]]$, odd under $t$ goes to $-t$, with lead term $t$, that satisfies any of a number of equivalent conditions. 

If $A$ has characteristic 0, one of these equivalent conditions 
characterizing $\sigma_{E/A}(t)$ is that $$D(D(\sigma_{E/A}(t))/\sigma_{E/A}(t))+x(t)\in A,$$ where $D$ acts on power series as in (\ref{Daction}). However, if $A$ has characteristic $p$, this condition does not uniquely characterize $\sigma_{E/A}$. On the other hand, Mazur and Tate show that for all complete DVRs $A$, $\sigma_{E/A}$ is uniquely characterized by the property that for all $u,v$ in the kernel of reduction $E_0(E/A)$,
$${{\sigma_{E/A}(u+_E v)\sigma_{E/A}(u-_Ev)}\over {\sigma_{E/A}(u)^2\sigma_{E/A}(v)^2}}=x(v)-x(u),$$ where $+_E$ and $-_E$ are denoting that the operations are taking place in
the group law of $E$.

We will now show that for our WOGEC $\cal E/\hat{R}$, that in an appropriate sense,
$\sigma_{\cal E/\hat{R}}$ universally satisfies this condition.

For parameters $t_1$ and $t_2,$ let $\cal F={\cal F}_{\cal E/\hat R}(t_1,t_2)$ be the formal group law in $\hat R[[t_1,t_2]]$ as in \cite{Si}, IV, \S1 for $\cal E/\hat R$, which we also write as $t_1+_{\cal F} t_2$,
 the power series gotten by calculating the expansion of $t$, in terms of $t_1$ and $t_2$, evaluated at the sum in the group law on $E$ of the points $(x(t_1),y(t_1))$ and $(x(t_2),y(t_2))$ of $E$. Then $\omega(t)=W(t)dt$ is an invariant differential on $\cal F$, i.e. [S, IV, \S 4], so
$$W(t_1+_{\cal F} t_2){d\over{dt_1}}(t_1+_{\cal F} t_2)=W(t_1),$$
and it follows that $D$ acts as an invariant derivation on $\cal F$, i.e. if $D_1$ denotes $D$ acting on $t_1$ while treating
$t_2$ as a constant, $$D_{1}(t_1+_{\cal F}t_2)={{d(t_1+_{\cal F}t_2)/dt_1}\over {W(t_1)}} =1/W(t_1+_{\cal F}t_2)=D(t)|_{t=t_1+_{\cal F}t_2}.$$ 
It follows from standard properties of derivations that for any power series $e\in \hat{R}[[t]]$ that 
$$D_1(e(t_1+_{\cal F}t_2))=D(e(t))|_{t=t_1+_{\cal F}t_2}.$$
We also write  $t_1-_{\cal F} t_2$ for subtraction in the formal group, which since $t$ is an odd parameter on $\cal E$, is the same as $t_1+_{\cal F}(-t_2)$, so we also have
$$D_{1}(e(t_1-_{\cal F}t_2))=D(e(t))|_{t=t_1-_{\cal F}t_2}.$$

\begin{proposition}\label{univtheoremofthesquare} 
As elements in the fraction field of $\hat{R}[[t_1,t_2]]$,
$${{\sigma_{\cal E/\hat{R}}(t_1+_{\cal F} t_2)\sigma_{\cal E/\hat{R}}(t_1-_{\cal F} t_2)}\over{\sigma_{\cal E/\hat{R}}^2(t_1)\sigma_{\cal E/\hat{R}}^2(t_2)}}=x(t_2)-x(t_1).$$
\end{proposition}

\begin{proof} 
Let $\sigma_{\cal E/\hat{R}}(t_1+_{\cal F} t_2)\sigma_{\cal E/\hat{R}}(t_1-_{\cal F} t_2)/\sigma_{\cal E/\hat{R}}^2(t_1)\sigma_{\cal E/\hat{R}}^2(t_2)=\theta(t_1,t_2).$
By Theorems \ref{theorem1} and \ref{theorem2} we have
$$D_1({{D_1(\sigma_{\cal E/\hat{R}}(t_1+_{\cal F} t_2))}\over {\sigma_{\cal E/\hat{R}}(t_1+_{\cal F} t_2)}})=D_1(\zeta_{\cal E/\hat{R}}(t_1+_{\cal F} t_2))=-x(t_1+_{\cal F} t_2)+\beta.$$
Applying this also with $t_2$ replaced by $-t_2$, then Theorem \ref{theorem2} and Lemma \ref{grouplaw} imply that 
the second logarithmic derivations in $t_1$
of $\theta(t_1,t_2)$ and $x(t_2)-x(t_1)$ agree, so there is an element $\mu(t_2)$ in the fraction field of $\hat{R}[[t_2]]$ such that 
$${{D_1\theta(t_1,t_2)}\over{\theta(t_1,t_2)}}-{{D_1(x(t_2)-x(t_1))}\over{x(t_2)-x(t_1)}}=\mu(t_2).$$
Since the lefthandside of this is odd in $t_1$, $\mu(t_2)=0$. Hence $$\theta(t_1,t_2)=\nu(t_2)(x(t_2)-x(t_1))$$ 
for some $\nu(t_2)$  in the fraction field of $\hat{R}[[t_2]]$.
Since $\theta(t_1,t_2)$ is odd under swapping $t_1$ and $t_2$, $\nu(t_2)=\nu(t_1)$ must be in $\hat{R}$.
Comparing the lead terms in the expansions of both sides of this as Laurent series in $t_1$ and $t_2$ shows that $\nu=1$.
\end{proof}

\begin{remark} One could also fashion a proof of the Proposition using the Lefschetz Principle and
properties of the complex sigma function.
\end{remark}


We now have one of our defining goals:

\begin{theorem}\label{specialize}
Let $p>3$ and $A$ be a complete discrete valuation ring of residue characteristic $p$, and $E$ an elliptic curve over $A$ in Weierstrass form (\ref{WeierModel}) with ordinary good (or multiplicative) reduction.  From Proposition ~\ref{specialization} there is a homomorphism
$\rho: \hat{R}\rightarrow A$ such that $\rho(A_4)=a_4$ and $\rho(A_6)=a_6$, which makes
$E$ a $\rho$-specialization $\cal E_\rho$.

Then the specialization $\widetilde{\sigma_{\cal E/\hat{R}}}$ 
of the universal $p$-adic sigma function $\sigma_{\cal E/\hat{R}}$ induced by $\rho$ 
is the Mazur-Tate $p$-adic sigma function $\sigma_{E/A}$.
\end{theorem}

\begin{proof} Note that if $\tilde{\cal F}$ is the  formal group law over $A$ 
gotten by specializing the coefficients of $\cal F$ via $\rho$,
then
$\tilde{\cal F}$ is a formal group law on the kernel of reduction $E_0(E/A)$ of $E/A$. 
Hence for any $u$ and $v$ in $E_0(E/A)$, the specialization $\hat{R}[[t_1,t_2]]\rightarrow A$ induced by $\rho$
and the map $t_1\rightarrow t(u)$, $t_2\rightarrow t(v)$, specialize the result of Proposition \ref{univtheoremofthesquare} 
to the equation,
$${{\widetilde{\sigma_{\cal E/\hat{R}}}(u+_E v)\widetilde{\sigma_{\cal E/\hat{R}}}(u-_E v)}\over{\widetilde{\sigma_{\cal E/\hat{R}}}^2(u)\widetilde{\sigma_{\cal E/\hat{R}}}^2(v)}}=x(v)-x(u),$$
where $x$, $y$, and $t=-x/y$ denote the functions on $E/A$ given in (\ref{WeierModel}).
Therefore by Theorem 3.1 of \cite{MT}, $\widetilde{\sigma_{\cal E/\hat{R}}}=\sigma_{E/A}$. 
\end{proof}

 


\section{Recovering the universal $p$-adic sigma functor.}

Now let $A$ be a complete DVR of residue characteristic $p>2$, and $E/A$ an elliptic curve with good ordinary or multiplicative reduction over $A$. Mazur and Tate constructed their $p$-adic sigma function for $E/A$ without the need to choose a model for $E$, defining it for a pair $(E,\omega)$ where 
$\omega$ is a choice of invariant differential on $E/A,$ and denoting it as $\sigma_{(E,\omega)/A}.$ 

As we noted in the Introduction, Mazur and Tate showed that their construction carried over to more general base schemes.

Let $\mathcal S$ denote the category of formal adic schemes for which $p$ can be taken as an ideal of definition.
For any $S\in\mathcal S$ and $n\geq 1$, let $S_n$ be the scheme cut out by the ideal generated by $p^n$. Then (see section 2 of \cite{BG}) an ordinary elliptic curve $E/S$ is a compatible system of ordinary elliptic curves $E_n$ over $S_n$ as $n$-varies.

Mazur and Tate constructed a  ``$\sigma$-functor'' for ordinary elliptic curves (along with a choice of non-vanishing relative 1-differential) over $\cal S$ which is uniquely determined by being compatible with base change, and by recovering their construction above for an elliptic curves with good ordinary reduction over a $p$-complete DVR $A$ (whose reductions mod $p^n$ can be viewed as an elliptic curve over $\Spf(A)$).

Let us recall what this functor does (for details see [MT]). 
For $S\in \mathcal S$, suppose $(E,\omega)$ is an ordinary elliptic curve over $S$ with $\omega$ a non-vanishing relative $1$-differential over $S$.  
Let $E^f_{/S}$ be  the  formal  completion  of $E$ along  the  zero-section  restricted  to $S_1$. 
They defined the sigma functor as a rule that assigned to each such $(E,\omega)$ a formal parameter $\sigma(E,\omega)_{/S}$ for the formal group $E^f_{/S}$ such that $d\sigma(E,\omega)_{/S}/\omega$ restricts to $1$ on the zero section of $E/S$.


 We will now sketch how our universal $p$-adic sigma function recovers the Mazur-Tate $\sigma$-functor when $p>3$. 
For starters, let $E$ be an ordinary elliptic curve over any scheme $S$ for which $p$ is nilpotent, and $\omega$ a choice of non-vanishing relative 1-differential over $S$. Let $U_i=\Spec(R_i)$
be an open cover of $S$, so $p$ is nilpotent in $R_i$ and hence $R_i$ is $p$-complete.
There is then a unique Weierstrass model $W_i$ for $E_{/{R_i}}$ of the form
$$y_i^2=x_i^3+\alpha_{4,i}x_i+\alpha_{6_i},$$ $\alpha_{4,i}, \alpha_{6,i}\in R_i$,
such that $\omega|_{U_i}=dx_i/2y_i.$ Let $t_i=-x_i/y_i$. Since such an elliptic curve is a WOGEC, by Proposition \ref{rho-specialization}, $W_i$ is  uniquely a $\rho$ specialization of $\cal E$, and applying $\rho$ to the coefficients
of $\sigma_{\mathcal E, \hat{R}}$ gives a power series $\sigma_i\in R_i[[t_i]]$ with lead term $t_i$.
In other words, if $E^f_{/S}$ is the formal completion of $E$ along the $0$-section, $\sigma_i$ is a parameter for the formal group $E^f_{/R_i}.$
By the uniqueness of these Weierstrass models, the corresponding power series agree on any overlap among the $U_i$, and the $\sigma_i$ therefore piece together to give a well-defined parameter $\hat{\sigma}(E,\omega)_{/S}$ for $E^f_{/S}$ such that $d\hat{\sigma}(E,\omega)_{/S}/\omega$ restricts to $1$ on the zero section of $E/S$.

Now take $S\in \cal S$, and let $(E,\omega)$ be an ordinary elliptic curve over $S$ with $\omega$ a non-vanishing relative 1-differential.
Then $E_n/S_n$
is an ordinary elliptic curve over a scheme where $p$ is nilpotent, and $\omega_n:=\omega|_{E_n}$ is a non-vanishing relative 1-differential, so the above defines a unique parameter $\hat{\sigma}_{(E_n,\omega_n)/S_n}$ for ${E_n^f}_{/S_n}.$
The uniqueness means that as $n$ varies they coherently define a unique parameter 
$\hat{\sigma}_{(E,\omega)/S}$ for $E^f_{/S}.$


 

That $\hat{\sigma}_{(E,\omega)/S}={\sigma}_{(E,\omega)/S}$ follows from the uniqueness of the $\sigma$-functor
and from Theorem \ref{specialize}, 
using that specialization commutes with base change.


\begin{thebibliography}{Cliff}
%
%
%
\bibitem[BB]{BB} J. S. Balakrishnan and A. Besser, \textit{Coleman-Gross height pairings and the $p$-adic sigma function.}
Journal f\"ur die reine und angewandte Mathematik (Crelles Journal) 698 (2015) 89--104.

\bibitem[BKY]{BKY} K. Bannai, S. Kobayashi, and S. Yasuda, \textit{The radius of convergence of the $p$-adic sigma function.}
Math. Z. {\bf 286}  (2017) 751--781.

\bibitem[BG]{BG}
J. Borger and L. Gurney, \textit{Canonical lifts of families of elliptic curves}, Nagoya Math. J., {\bf 233} (2019), 193--213. 
%
\bibitem[Bl]{Bl} C. Blakestad, \textit{On Generalizations of $p$-Adic Weierstrass Sigma and Zeta Functions,} PhD Thesis, University of Colorado Boulder (2018).
%
\bibitem[Br]{Br}
L. Breen, {\textit{Fonctions Th\^eta et th\'eor\`eme du cube,}} Lecture Notes in Math. 980, Springer-Verlag, Berlin, 1983.
%
\bibitem[Cas]{Cas} J. W. S. Cassels, \textit{A note of the division values of $\wp(u)$}, 
Math. Proc. Camb. Phil. Soc. 45, No. 2 (1949) 167--172.
%


\bibitem[Co1]{Co1}
B. Conrad, {\textit{Arithmetic moduli of generalized elliptic curves,}}
Journal of the Institute of Mathematics of Jussieu 6 (2007) 209--278.
%
\bibitem[Co2]{Co2}
B. Conrad, \textit{Math 248B. Modular Curves}, Stanford University course handout:

 \noindent
  http://virtualmath1.stanford.edu/$\sim$conrad/248BPage/handouts/modularcurves.pdf
%
\bibitem[Cr1]{Cr1}
V. Cristante, {\textit{Theta functions and Barsotti-Tate groups,}} Ann. Scuola Norm. Sup. Pisa C1.
Sci. (4) 7 (1980), 181--215.
%
\bibitem[Cr2]{Cr2}
V. Cristante, {\textit{p-Adic theta series with integral coefficients,}} (1984), 169--182.
%
\bibitem[Der]{Der}
C. Derby,
{\textit{
Beyond two criteria for supersingularity: coefficients of division polynomials,
}}
Journal de Th\'eorie des Nombres de Bordeaux 26 (2014), 595-605.
%
\bibitem[Deu]{Deu}
M. Deuring, {\textit{Die Typen der Multiplikatorenringe elliptischer  Funktionk\"orper,}}
Abh. Math. Sem. Hamburg 14 (1941), 197--272.
%
%
\bibitem[Elk]{Elk}
N. Elkies, {\textit{Elliptic and modular curves over finite fields and related computational issues,}}  in Computational Perspectives on Number Theory: Proceedings of a Conference in Honor of A.O.L. Atkin (D. A. Buell and J. T. Teitelbaum, eds.; AMS/International Press) (1998)
21--76.
%
\bibitem[Ell]{Ell}
J. Elliott,
{\textit{
Factoring formal power series over principal ideal domains,
}}
Journal de Th\'eorie des Nombres de Bordeaux 26, 595-605 (2014)
%
\bibitem[G]{G}
F. Q. Gouvea,
{\textit{Arithmetic of p-adic Modular Forms,}}
Lecture Notes in Mathematics 1304, Springer-Verlag, 1988. 
%
%
\bibitem[Has]{Has}
H. Hasse,
{\textit{Existenz separabler zyklischer unverzweigter Erweiterungsk\"orper vom Primzahlgrade $p$ \"uber elliptischen Funktionk\"orpern
der Charakteristik $p$,}}
Journal f\"ur die reine und angewandte Mathematik (Crelles Journal) 172 (1934) 77--85.
%
\bibitem[Haz]{Haz}
M. Hazewinkel,
{\textit{
Formal Groups and Applications,
}}
Academic Press 1978. 
%
\bibitem[Hi]{Hi}
H. Hida, {\textit{Geometric Modular Forms and Elliptic Curves,}} 
World Scientific, Singapore, 2011.
%
\bibitem[Ho]{Ho}
T. Honda, {\textit{On the theory of commutative formal groups,}} J. Math. Soc. Japan, 
Vol 22, No. 2 (1970),  213--246.
%
\bibitem[K]{K}
N. Katz, \textit{$p$-adic Properties of Modular Schemes and Modular Forms},  in: W. Kuijk, J-P. Serre (eds), Modular Functions of One Variable III. Lecture Notes in Mathematics, vol 350. Springer, Berlin, Heidelberg (1973).
%
\bibitem[KM]{KM}
N. Katz and B. Mazur, {\textit{Arithmetic Moduli of Elliptic Curves,}} Annals of Mathematics Studies, Volume 108.
Princeton, 1985.
%
\bibitem[L]{L}
S. Lang,
{\textit{Elliptic Functions,}}
Second edition. Graduate Texts in Mathematics, 112. Springer-Verlag, 1987
%
%
\bibitem[MST]{MST}
B. Mazur, W. Stein, J. Tate,
{\textit{Computation of p-Adic Heights and Log Convergence,}}
Documenta Math. (Extra Volume: John H.  Coates' 60th Birthday) (2006) 577--614.
%
\bibitem[MT]{MT}
B. Mazur, J. Tate,
{\textit{The p-adic sigma function,}}
Duke Math. J. 62 (1991), no. 3, 663--688.
%
\bibitem[MTT]{MTT}
B. Mazur, J. Tate, J. Teitelbaum, {\textit{On p-adic analogues of the conjectures
of Birch and Swinnerton-Dyer,}}  Invent. math. 84 (1986) 1--48.
%
\bibitem[N]{N}
P. Norman,
{\textit{p-adic theta functions,}}
American Journal of Mathematics 107 (1985), 617--661.
%
\bibitem[O]{O}
M. O'Malley,
{\textit{
On the Weierstrass preparation theorem,
}}
Rocky Mountain J. Math., 2(2), (1972), 265--273.
%
%
%
%
%
\bibitem[P]{P} M. Papanikolas, \textit{Canonical heights on elliptic curves in characteristic $p$,}
 Compositio Mathematica, 122(3), (2000),  299--313.
%
\bibitem[Se]{Se}
J-P. Serre, \textit{Formes modulaires et fonctions z\^eta $p$-adiques},  in: W. Kuijk, J-P. Serre (eds), Modular Functions of One Variable III. Lecture Notes in Mathematics, vol 350. Springer, Berlin, Heidelberg (1973).
%
\bibitem[Si]{Si}
J. H. Silverman,
{\textit{The arithmetic of elliptic curves,}}
Graduate Texts in Mathematics 106, Springer-Verlag, 2000. 
%
\bibitem[SB]{SB} J. Stienstra, F. Beukers, \textit{On the Picard-Fuchs Equation and the Formal Brauer Group of Certain Elliptic K3-Surfaces}, Math. Ann., 271 (1985) 269--304.
%
%
\bibitem[T]{T} J. Tate, \textit{A review of non-Archimedean elliptic functions,}
 in Elliptic Curves, Modular forms, \& Fermat’s Last Theorem (Hong Kong, 1993), 162--184, Int. Press, Cambridge, MA, 1995.

 %
 \bibitem[Ve]{Ve} J. V\'elu, \textit{Isog\'enies entre courbes elliptiques,} C. R. Acad. Sci. Paris S\'er. A-B 273 (1971) A238-A241.
 
 %
 \bibitem[Vo1]{Vo1}
 J. F. Voloch,  \textit{Explicit $p$-descent for elliptic curves in characteristic $p$},
 Compositio Mathematica, 74 (1990), no. 3 247--258.
 
 %
 \bibitem[Vo2]{Vo2}
 J. F. Voloch,  \textit{An analogue of the Weierstrass $\zeta$-function in characteristic $p$},
 Acta Arithmetica, LXXIX.1 (1997) 1--6.
 
 %
%
\bibitem[Was]{Was}
L. Washington,
{\textit{Elliptic curves: number theory and cryptography,}}
Second edition. Chapman \& Hall/CRC, New York, 2008. 
%

\end{thebibliography}
\end{document}